\documentclass[11pt]{amsart}
\usepackage{amsmath,amssymb,amsthm,latexsym,cite,cancel}
\usepackage[small]{caption}
\usepackage{graphicx,wasysym,overpic,tikz,color}
\usepackage{subfigure}
\usepackage{cite}
\usepackage[colorlinks=true,urlcolor=blue,
citecolor=red,linkcolor=blue,linktocpage,pdfpagelabels,
bookmarksnumbered,bookmarksopen]{hyperref}
\usepackage[english]{babel}
\usepackage{units}
\usepackage{enumitem}
\usepackage[left=2.1cm,right=2.1cm,top=2.71cm,bottom=2.71cm]{geometry}
\usepackage[hyperpageref]{backref}
\usepackage{float,bm}

\usepackage[colorinlistoftodos]{todonotes}
\newtheorem{theorem}{Theorem}[section]
\newtheorem{definition}[theorem]{Definition}
\newtheorem{proposition}[theorem]{Proposition}
\newtheorem{lemma}[theorem]{Lemma}

\newcommand{\R}{\mathbb{R}}

\tikzstyle{nodo}=[circle,draw,fill,inner sep=0pt,minimum size=%
1.5mm]

\numberwithin{equation}{section}

\author[K. Sheng]{kai sheng}

\address[K. Sheng]{\newline\indent
	School of Mathematics
	\newline\indent
	East China University of Science and Technology
	\newline\indent
	Shanghai 200237, PR China }
\email{\href{mailto:shengkai2001@outlook.com}{shengkai2001@outlook.com}}

\title[SBP SYSTEMS IN BOUNDED DOMAINS]{Normalized solutions for Schrödinger-Bopp-Podolsky systems in bounded domains with general nonlinearities}


\date{\today}

\begin{document}

\begin{abstract}
	In this paper, we study normalized standing wave solutions for the following nonlinear Schr\"odinger--Bopp--Podolsky system:  
	\[
	\begin{cases}
		-\Delta u + q(x)\,\phi\,u = \omega u + f(u) & \text{in }\Omega,\\[2mm]
		-\Delta \phi + a^{2}\Delta^{2}\phi = q(x)u^{2} & \text{in }\Omega,
	\end{cases}
	\]
	where $\Omega \subset \mathbb{R}^3$ is a smooth bounded domain, $a>0$, $\omega\in\mathbb{R}$ is the Lagrange multiplier associated with the $L^{2}$-mass constraint $\int_{\Omega} u^2 \, dx = \mu^2$, and $f:\mathbb{R}\to\mathbb{R}$ is a continuous function satisfying some technical conditions. We introduce a perturbation framework for the problem and investigate normalized solutions. In particular, we prove the existence of normalized solutions for all masses $\mu \in (0,\mu_{0})$ under either Navier or Neumann boundary conditions for $\phi$. Moreover, multiplicity of normalized solutions is obtained when $f$ is odd. Based on these facts, we further prove that if $\Omega$ is star-shaped, then there exists a normalized ground state solution.
\end{abstract}
\maketitle

\section{Introduction}

\subsection{Background on the Bopp--Podolsky theory}

The Bopp--Podolsky theory, introduced independently by Bopp \cite{Bopp1940} and Podolsky \cite{Podolsky1942}, is a second-order gauge theory for the electromagnetic field and was proposed in order to overcome the so-called infinity problem arising in the classical Maxwell model. For further results of this framework, we refer to \cite{Bonin2019,Bertin2017,Bufalo2014,Bufalo2017,Cuzinatto2018,Cuzinatto2017}.

 In the classical setting, the electrostatic potential $\phi$ generated by a charge distribution with density $\rho$ solves the Poisson equation
\[
	-\Delta \phi = \rho \quad \text{in } \R^{3},
\]
and, when $\rho=4\pi\delta_{x_0}$ is a point charge located at $x_0\in\R^{3}$, one has $\phi(x)=\frac{1}{|x-x_0|}$ and the corresponding electrostatic energy is infinite:
\[
\int_{\R^{3}} |\nabla \phi|^{2}\,dx = +\infty.
\]
On the other hand, Bopp and Podolsky proposed to replace the Poisson equation by 
\[
	-\Delta \phi + a^{2} \Delta^{2} \phi = \rho \quad \text{in } \R^{3},
\]
where $a>0$ is the Bopp--Podolsky parameter.
For the same point charge $\rho=4\pi\delta_{x_0}$ the explicit solution of $-\Delta \phi + a^{2} \Delta^{2} \phi = \rho \quad \text{in } \R^{3}$ is given by 
\[
\mathcal K(x)=\frac{1-e^{-|x|/a}}{|x|}, \qquad x\in\R^{3},
\]
which is finite at the origin and satisfies 
\[
\int_{\R^{3}} |\nabla \mathcal K|^{2}\,dx + a^{2}\int_{\R^{3}} |\Delta \mathcal K|^{2}\,dx < +\infty.
\]
Hence, under the Bopp--Podolsky framework, the electrostatic field generated by a point charge has finite energy in the BP norm $\int_{\mathbb{R}^3} (|\nabla\phi|^{2} + a^{2}|\Delta\phi|^{2})\,dx$.

\subsection{Motivation and main results}

In the electrostatic framework, it is natural to couple the Bopp--Podolsky equation  with a nonlinear Schrödinger equation (see \cite{Siciliano2025,dAveniaSiciliano2019}). This coupling yields the Schrödinger--Bopp--Podolsky systems, which can be viewed as generalized versions of the classical Schrödinger--Poisson or Schrödinger--Maxwell models.

For the whole-space problem with fixed frequency $\omega$, a variety of situations has already been studied, in particular, \cite{MascaroSiciliano2023} established the existence of positive solutions for the following Schrödinger--Bopp--Podolsky system in $\mathbb{R}^3$. 
For other related results, we refer to \cite{dAveniaSiciliano2019, Hebey2019, RamosdePaulaSiciliano2023, SantosDamianSicilianoPreprint, FigueiredoSiciliano2023,LiPucciTang2020,Caponio2025,SicilianoSilva2020}. 
\[
	\begin{cases}
		-\Delta u + q(x)\,\phi\,u \;=\; \omega u + f(u) & \text{in }\mathbb{R}^3,\\[2mm]
		-\Delta \phi + a^{2}\Delta^{2}\phi \;=\; q(x)\,u^{2} & \text{in }\mathbb{R}^3.
	\end{cases}
\]
In this paper, we investigate the following Schrödinger--Bopp--Podolsky system in bounded domains, a setting that is  physically relevant for confined devices (see \cite{Siciliano2025}). In this regime, the wave function $u$ is defined in a bounded conductor, and the electrostatic potential $\phi$ satisfies appropriate boundary conditions on $\partial\Omega$. More precisely, we consider the following system:
\[
\begin{cases}
	-\Delta u + q(x)\,\phi\,u = \omega u + f(u), & \text{in } \Omega,\\[2mm]
	-\Delta \phi + a^{2}\Delta^{2}\phi = q(x)\,u^{2}, & \text{in } \Omega,\\[2mm]
	u = 0, & \text{on } \partial\Omega,\\[2mm]
	\displaystyle \int_{\Omega} u^{2}\,dx = \mu.
\end{cases}
\]

where $\Omega \subset \mathbb{R}^3$ is a smooth bounded domain, $a>0$, $\mu>0$ is prescribed, and $\omega\in\mathbb{R}$ is part of the unknown which appears as a Lagrange multiplier. The Dirichlet boundary condition enforces confinement of the particle, while the $L^2$-constraint  prescribes its mass.

The above problem has been proposed by \cite{Siciliano2025}, obtained infinitely many normalized solutions for both Navier and Neumann boundary conditions by means of a Ljusternik--Schnirelmann type argument, for $f(u)=|u|^{p-2}u$ with $2<p<\tfrac{10}{3}$ (i.e. in the $L^{2}$-subcritical regime).

To the best of our knowledge, the normalized problem for this system on bounded domains in the regime $\tfrac{10}{3}<p<6$ (i.e. in the $L^{2}$-supercritical regime) has not been established yet. In this regime the associated energy functional is no longer bounded from below, and on bounded domains the scaling arguments available in $\R^{3}$ cannot be used. In this paper, we work with a general nonlinear term $f(u)$, which in particular covers the $f(u)=|u|^{p-2}u$ with $\tfrac{10}{3}<p<6$ .

 Since \(\Omega \subset \mathbb{R}^{3}\) is bounded, the Sobolev embeddings
\(H_{0}^{1}(\Omega)\hookrightarrow L^{p}(\Omega)\) are compact, hence we are no longer faced with the lack of compactness that appears on
\(\mathbb{R}^{3}\). In particular, in the \(L^{2}\)-subcritical case, the  existence analysis becomes
considerably simpler.
In contrast, in the $L^{2}$-supercritical case, the usual scaling arguments fail on bounded domains, and Pohozaev-type identities produce boundary terms that require quantitative control.
To overcome these difficulties, several works on normalized solutions in bounded domains have used more delicate tools, such as blow-up analysis, Morse/Ljusternik--Schnirelmann theory and Poho\v{z}aev identities. See, for instance, Pierotti and Verzini~\cite{PierottiVerzini2017}, where existence for small masses under $L^{2}$–supercritical growth is obtained. More recently, Bartsch, Qi and Zou~\cite{BartschQiZou2024} combined the monotonicity trick with a Poho\v{z}aev identity to treat a Schr\"odinger equation with two competing powers, and proved the existence of positive normalized solutions together with uniform bounds on the corresponding Lagrange multipliers. In a different direction, Song and Zou~\cite{SongZou2025} constructed a new linking on a Hilbert--Riemannian manifold below a suitable energy level and derived multiple sign-changing normalized solutions for a Br\'ezis--Nirenberg type problem on bounded domains.

The above approaches require at least $f\in C^{1}(\mathbb{R})$. Hence they do not apply when $f\in C(\mathbb{R},\mathbb{R})$ only. In \cite{Alves-He-Ji2025}, for the problem
\[
\begin{cases}
	-\Delta u = \lambda u + f(u) & \text{in } \Omega,\\[2pt]
	u = 0 & \text{on } \partial\Omega,\\[2pt]
	\displaystyle \int_{\Omega} u^{2}\,dx = \mu,
\end{cases}
\]
where $\Omega \subset \mathbb{R}^{N}$  is a smooth bounded domain and $f\in C(\mathbb{R},\mathbb{R})$, using a perturbation method together with the energy estimate
\[
E(u) \le \tfrac12\,\lambda_{1}\,\mu,
\]
Alves, He and Ji first established a nonexistence theorem for normalized solutions under different boundary conditions, and then, derived corresponding existence and multiplicity results. Moreover, when the domain is star-shaped ($N\ge3$), they further applied the Poho\v{z}aev identity to obtain a normalized ground state.

Inspired  by \cite{Siciliano2025,Alves-He-Ji2025,Buffoni-Esteban-Sere2006}, we consider the following Schr\"odinger--Bopp--Podolsky system  on  bounded domains under two different boundary conditions:

\begin{equation}\label{1.1}
	\begin{cases}
		-\Delta u + q(x)\,\phi\,u = \omega u + f(u), & \text{in } \Omega,\\[2mm]
		-\Delta \phi + a^{2}\Delta^{2}\phi = q(x)\,u^{2}, & \text{in } \Omega,\\[2mm]
		u = 0, & \text{on } \partial\Omega,\\[2mm]
		\displaystyle \int_{\Omega} u^{2}\,dx = \mu, &
	\end{cases}
\end{equation}
where $\Omega\subset\mathbb{R}^3$ is a smooth bounded domain, $\omega\in\mathbb{R}$ is a part of the unknown which appears as a Lagrange multiplier, $\mu>0$ is prescribed, and $q\in C(\overline{\Omega})\setminus\{0\}$ is a given datum representing a nonuniform charge distribution. For simplicity, we consider the case $a=1$.

\noindent
\textbf{(I) Homogeneous Navier boundary conditions:}
\begin{align}
	\phi &= 0 \quad \text{on }\partial\Omega, \label{1.2}\\
	\Delta\phi &= 0 \quad \text{on }\partial\Omega. \label{1.3}
\end{align}
The reason for which we have chosen zero boundary conditions is just technical.

\medskip
\noindent
\textbf{(II) Neumann boundary conditions}
\begin{align}
	\frac{\partial \phi}{\partial n} &= h_{1} \quad \text{on }\partial\Omega, \label{1.4}\\
	\frac{\partial \Delta \phi}{\partial n} &= h_{2} \quad \text{on }\partial\Omega, \label{1.5}
\end{align}
where, for simplicity, $h_{1},h_{2}$ are continuous functions defined on $\partial\Omega$. The symbol $\nu$ denotes the unit outward normal to $\partial\Omega$. Physically, these conditions prescribe the flux of the electric field across the boundary.

From a variational viewpoint, system \eqref{1.1} under \eqref{1.2}--\eqref{1.3} is naturally associated with the functional 
\[
\mathcal F(u,\phi)
= \frac12 \int_{\Omega} |\nabla u|^{2} \,dx
+ \frac12 \int_{\Omega} q(x)\,\phi\,u^{2} \,dx
- \int_{\Omega} F(u)\,dx
- \frac14 \int_{\Omega} |\Delta \phi|^{2}\,dx
- \frac14 \int_{\Omega} |\nabla \phi|^{2}\,dx,
\]
defined on $H_{0}^{1}(\Omega)\times H^{2}(\Omega)$. 
However, this functional is unbounded from above and from below, so the usual critical point arguments cannot be applied directly. 
Motivated by the reduction procedure used in \cite{BenciFortunato1998,Siciliano2025}, we first solve the second equation in \eqref{1.1} to obtain $\phi=\Phi(u)$ (see\eqref{2.7}) and then introduce a single-variable functional depending only on $u$, see\eqref{2.9}.
In the reduced functional \eqref{2.9}, the contribution of $\phi$ appears through the term
\begin{equation} \label{estimate}
	\|\Phi(u)\|_{H_0^1(\Omega)\times H^2(\Omega)}^{2}
	= \int_{\Omega} \big(|\nabla \Phi(u)|^{2} + |\Delta \Phi(u)|^{2}\big)\,dx
	\;\le\; C\,\|u\|_{H_0^1(\Omega)}^{4}.
\end{equation}
At present, only continuity-type estimates are available for \eqref{estimate}. 
In the $L^{2}$-supercritical range $10/3 < p < 6$, \eqref{estimate} is not strong enough to recover the Cerami compactness, 
as the boundedness step becomes significantly more delicate. 
Moreover, \eqref{estimate} does not yield an energy bound of the form
\[
J(u) \le \frac12\,\lambda_{1}\,\mu .
\]
Hence, the perturbation method of \cite{Alves-He-Ji2025} cannot be applied directly.
To overcome these difficulties, we prove in Lemma~\ref{lem:Phi-mu} a uniform upper estimate for the induced potential $\Phi(u)$ for any fixed $\mu>0$.
This estimate will be used throughout the paper to control $\|\Phi(u)\|_{H_0^1(\Omega)\times H^2(\Omega)}$.

We now state the assumptions on the nonlinearity $f:\R\to\R$.

\medskip
\noindent
\textbf{($f_1$)} There exist constants $2<p<6$ and $0\le K_{2}<\lambda_{1}(\Omega)$, $K_{p}\ge0$ such that 
\[
|f(t)| \le K_{2}|t| + K_{p}|t|^{p-1} \quad \text{for all } t\in\R,
\]
where $\lambda_1 := \lambda_{1}(\Omega)$ denotes the first Dirichlet eigenvalue of $-\Delta$ in $\Omega$. 

\noindent
\textbf{($f_2$)} $f(t)t\ge0$ for all $t\in\R$ and the map $t\mapsto \dfrac{f(t)}{|t|}$ is nondecreasing on $(-\infty,0)\cup(0,+\infty)$.

\noindent
\textbf{($f_3$)} $\displaystyle \lim_{|t|\to +\infty} \frac{F(t)}{t^{2}} = +\infty$, where $F(t):=\int_{0}^{t} f(s)\,ds.$

\noindent
\textbf{($f_4$)} There exists $q>2$ such that $0 \le q\,F(t) \le f(t)t \quad \text{for all } t\in\R.$

\medskip

We point out that assumptions \((f_1)\)–\((f_4)\) include many standard nonlinearities. For instance,
\[
f(t) = |t|^{s-2}t \quad \text{with } 2 < s < 6,
\]
\[
f(t) = |t|^{s_1-2}t + |t|^{s_2-2}t \quad \text{with } 2 < s_1,s_2 < 6,
\]
\[
f(t) = |t|^{s_1-2}t + |t|^{s_2-2}t \ln(1+|t|) \quad \text{with } 2 < s_1,s_2 < 6.
\]
It is worth recalling that assumption~($f_4$) corresponds to the classical Ambrosetti--Rabinowitz (AR) condition, 
which is often imposed to guarantee the boundedness of Palais–Smale sequences and the validity of the mountain pass geometry.  
However, the framework developed in this paper does not rely on the AR condition.  
In fact, we are able to deal with nonlinearities satisfying~($f_1$)--($f_3$) but violating~($f_4$).  
For instance, the following nonlinearity
\[
f(t)=
\begin{cases}
	|t|^{s_1-2}t, & \text{if } t\in[-1,1],\\[2pt]
	t, & \text{if } t\in[-2,-1)\cup(1,2],\\[2pt]
	|t-1|^{s_2-2}t, & \text{if } t\in\mathbb{R}\setminus[-2,2],
\end{cases}
\qquad \text{with } 2<s_1,s_2<6 .
\]
fails to satisfy the AR condition, yet it fits perfectly under our variational framework \((f_1)\)–\((f_3)\).

We are now in the position to state our main theorems. 

\begin{theorem}[Existence result with boundary conditions (I)]\label{thm:1.1}
	Suppose that \emph{($f_1$)--($f_3$)} hold or \emph{($f_1$) and ($f_4$)} hold. 
	Then there exists $\mu_{0}>0$, depending only on $\Omega$, $p$, and $q(x)$, such that, for every $0<\mu<\mu_{0}$ problem \eqref{1.1} 	under  \eqref{1.2}--\eqref{1.3} admits a solution 
\[
(u,\phi,\omega)\in H_{0}^{1}(\Omega)\times H^{2}(\Omega)\times \bigl(0,\ \lambda_{1}+C_{\Phi}^{2}\,\mu\,\bigr].
\]

\end{theorem}

For the second boundary setting we introduce the compatibility quantity 
\begin{equation}\label{1.9}
	\alpha := \int_{\partial\Omega} h_{2}\,ds - \int_{\partial\Omega} h_{1}\,ds,
\end{equation}
and we assume that 
\[
\inf_{\Omega} q < \alpha < \sup_{\Omega} q
\quad\text{and}\quad
|q^{-1}(\alpha)| = 0.
\]

\begin{theorem}[Existence result with boundary conditions (II)]\label{thm:1.2}
	Suppose that \emph{($f_1$)--($f_3$)} hold or \emph{($f_1$) and ($f_4$)} hold and that the above assumptions on $\alpha$ are satisfied. 
	Then there exists $\mu_{0}>0$, depending only on $\Omega$, $p$, and $q(x)$, such that, for every $0<\mu<\mu_{0}$ problem \eqref{1.1} under  \eqref{1.4}--\eqref{1.5} admits a solution 
	\[
	(u,\phi,\omega)\in H_{0}^{1}(\Omega)\times H^{2}(\Omega)\times (0,+\infty).
	\]
	
\end{theorem}

\medskip\noindent\textbf{Remark.}
The existence results corresponding to case~\textbf{(II)} are analogous to those obtained in case~\textbf{(I)}, 
for the simplicity, the proofs are therefore omitted. 
The only difference lies in that the way the energy functional is reduced to a single-variable formulation involving $u$, for more detail, see \cite{Siciliano2025},
while the overall variational structure and critical point analysis remain unchanged.

The perturbation approach also yields multiplicity when the nonlinearity is odd. 

\begin{theorem}\label{thm:1.3}
	Suppose that \emph{($f_1$)--($f_3$)} hold or that \emph{($f_1$) and ($f_4$)} hold, and assume in addition that $f(t)=-f(-t)$ for all $t\in\R$. 
	Then, for every $m\in\mathbb{N}$, there exists $\mu_{m}>0$, depending only on $\Omega$, $p$, $m$, and $q(x)$, such that, for every $0<\mu<\mu_{m}$ problem \eqref{1.1} under \eqref{1.2}--\eqref{1.3} admits at least $m$ distinct nontrivial solutions
	\[
	(u_{1},\phi_{1},\lambda_{1}),\; \dots,\; (u_{m},\phi_{m},\lambda_{m}) \in H_{0}^{1}(\Omega)\times H^{2}(\Omega)\times (0,+\infty).
	\]

\end{theorem}

Finally, by exploiting the Pohozaev-type identity tailored to \eqref{1.1} we obtain the existence of a ground state on star-shaped domains. 

\begin{theorem}\label{thm:1.4}
	Let $\Omega\subset\R^{3}$ be a smooth bounded domain, star-shaped with respect to the origin. 
	Suppose that \emph{($f_1$)}, \emph{($f_2$)} and \emph{($f_4$)} hold with $q>\tfrac{10}{3}$. 
	Then there exists $\mu_{*}>0$ such that for every $0<\mu\le \mu_{*}$ problem \eqref{1.1} under  \eqref{1.2}--\eqref{1.3} has a normalized ground state solution $(u,\phi,\lambda_{u})$. 
\end{theorem}

The rest of the paper is organized as follows.  
In Section~\ref{2}, we establish the variational framework associated with problem~\eqref{1.1},  
reduce the natural energy functional to a single-variable functional depending only on $u$,  
and derive refined upper estimates for the norm  $\|\Phi(u)\|_{H_0^1(\Omega)\times H^2(\Omega)}$ of the electrostatic potential.  
Section~\ref{3} deals with the existence and multiplicity of normalized solutions,  
which are obtained by perturbation variational methods and minimax arguments.  
In Section~\ref{4}, under the additional assumption that $\Omega$ is star-shaped,  
we derive a Poho\v{z}aev-type identity that is crucial in the analysis  
and prove the existence of a normalized ground state solution.

\section{Preliminaries }\label{2}
In this section we fix notation and recall the variational framework for the Schr\"odinger--Bopp--Podolsky system. 
We first introduce the natural energy functional on $H_0^1(\Omega)\times \mathbb{H}$, where $\mathbb{H}:=H_0^1(\Omega)\cap H^2(\Omega)$ and identify weak solutions as constrained critical points on the $L^2$-sphere. We then eliminate the electrostatic variable $\phi$ by solving the associated linear Bopp--Podolsky equation and obtain a reduced functional $J$ defined on $H_0^1(\Omega)$. 
Finally, for each fixed mass $\mu>0$, we establish an upper bound for the induced potential $\Phi(u)$ in the norm $\|\cdot\|_{H_0^1(\Omega)\times H^2(\Omega)}$, improving the basic continuity estimate derived from \eqref{2.9}.

\subsection{Notation and functional spaces}

Throughout the rest of the paper $\Omega\subset\R^{3}$ denotes a smooth bounded domain.
We denote $H^{2}(\Omega)$ and $H_{0}^{1}(\Omega)$ be the usual Sobolev spaces. 
In particular,
\[
(u,v) := \int_{\Omega} \nabla u \cdot \nabla v \,dx, \qquad
\|u\| := \bigg( \int_{\Omega} |\nabla u|^{2} \,dx \bigg)^{1/2},
\]
denote, respectively, the $H_0^1(\Omega)$ inner product and norm.

The $L^{2}(\Omega)$–inner product is given by 
\[
(u,v)_{2} := \int_{\Omega} u v \,dx.
\]
For $1\le p \le 6$ we write
\[
|u|_{p} := \bigg( \int_{\Omega} |u|^{p} \,dx \bigg)^{1/p}
\]
for the norm in $L^{p}(\Omega)$, and we denote by $|u|_{\infty}$ the norm in $L^{\infty}(\Omega)$. 
Henceforth we omit the spatial variable $x\in\Omega$ in all functions, except for the coefficient $q(x)$, in order to emphasize that it is nonconstant.

It is well known that if $\phi=\Delta\phi=0$ on $\partial\Omega$, then there exists $M=M(\Omega)>0$ such that
\[
\int_\Omega |D^{2}\phi|^{2}\,dx \;\le\; M \int_\Omega |\Delta\phi|^{2}\,dx
\qquad \forall\,\phi\in H_0^1(\Omega)\cap H^{2}(\Omega).
\]
We denote
\[
\mathbb{H}:=H_0^{1}(\Omega)\cap H^{2}(\Omega),\qquad
\|\phi\|_{\mathbb{H}}
:=\Big(\int_{\Omega}|\Delta\phi|^{2}\,dx + \int_{\Omega}|\nabla\phi|^{2}\,dx\Big)^{1/2},
\]
with inner product
\[
(\phi,\psi)_{\mathbb{H}}
:=\int_{\Omega}\Delta\phi\,\Delta\psi\,dx + \int_{\Omega}\nabla\phi\cdot\nabla\psi\,dx .
\]

For $\mu>0$ we introduce the $L^{2}$-sphere
\[
S(\mu):=\{\,u\in H_0^{1}(\Omega):\, |u|_{2}^{2}=\mu\,\}.
\]
As discussed in the introduction, $\omega$ will play the role of the Lagrange multiplier associated with this constraint, $S(\mu)$ will be the natural constraint for our variational problem.

\subsection{The variational setting}

In this subsection we work under the Navier boundary conditions \eqref{1.2}–\eqref{1.3}.

We say that a triple $(u,\omega,\phi)\in H_{0}^{1}(\Omega)\times \R \times \mathbb H$
is a weak solution of \eqref{1.1} under \eqref{1.2}–\eqref{1.3} if 
\begin{equation}\label{2.3}
	\int_{\Omega} \nabla u \cdot \nabla v \,dx
	+ \int_{\Omega} q(x)\,\phi\,u\,v \,dx
	- \int_{\Omega} f(u)\,v \,dx
	= \omega \int_{\Omega} u\,v \,dx,
	\qquad \forall\, v \in H_{0}^{1}(\Omega).
\end{equation}
and
\begin{equation}\label{2.4}
	\int_{\Omega} \Delta \phi \, \Delta \xi \,dx
	+ \int_{\Omega} \nabla \phi \cdot \nabla \xi \,dx
	= \int_{\Omega} q(x)\,u^{2}\,\xi \,dx,
	\qquad \forall\, \xi \in \mathbb H.
\end{equation}

 We introduce the natural energy functional
$\mathcal F : H_{0}^{1}(\Omega)\times \mathbb H \longrightarrow \R$
defined by 
\begin{equation}\label{2.5}
	\mathcal F(u,\phi)
	:= \frac12 \int_{\Omega} |\nabla u|^{2} \,dx
	+ \frac12 \int_{\Omega} q(x)\,\phi\,u^{2} \,dx
	- \int_{\Omega} F(u) \,dx
	- \frac14 \int_{\Omega} |\Delta \phi|^{2} \,dx
	- \frac14 \int_{\Omega} |\nabla \phi|^{2} \,dx,
\end{equation}
where $F(t)=\int_{0}^{t} f(s)\,ds$. 
A direct computation shows that $\mathcal F\in C^{1}(H_{0}^{1}(\Omega)\times \mathbb H,\R)$ and that its partial derivatives are given by 
\begin{align}
	\partial_{u}\mathcal F(u,\phi)[v]
	&= \int_{\Omega} \nabla u \cdot \nabla v \,dx
	+ \int_{\Omega} q(x)\,\phi\,u\,v \,dx
	- \int_{\Omega} f(u)\,v \,dx,
	\qquad \forall\, v \in H_{0}^{1}(\Omega), \label{2.6}\\[2mm]
	\partial_{\phi}\mathcal F(u,\phi)[\xi]
	&= \frac12 \int_{\Omega} q(x)\,u^{2}\,\xi \,dx
	- \frac12 \int_{\Omega} \Delta \phi \, \Delta \xi \,dx
	- \frac12 \int_{\Omega} \nabla \phi \cdot \nabla \xi \,dx,
	\qquad \forall\, \xi \in \mathbb H. \label{2.7}
\end{align}

\begin{theorem}\label{thm:2.1}
	Let $(u,\omega,\phi)\in H_{0}^{1}(\Omega)\times\R\times \mathbb H$. Then $(u,\omega,\phi)$ is a weak solution of \eqref{1.1} under the conditions \eqref{1.2}–\eqref{1.3} if and only if $(u,\phi)$ is a critical point of the functional $\mathcal F$ in \eqref{2.5} constrained to $S(\mu)\times \mathbb H$, where $\omega$ is the Lagrange multiplier associated with the constraint. 
\end{theorem}
As a consequence, critical points of $\mathcal F$ restricted to the set 
\[
S(\mu)\times \mathbb H = \{(u,\phi)\in H_{0}^{1}(\Omega)\times \mathbb H : |u|_{2}^{2}=\mu\}
\]
provides a weak solution of \eqref{1.1} under \eqref{1.2}–\eqref{1.3}, the real number $\omega$ in \eqref{2.3} being the corresponding Lagrange multiplier. 
This is completely analogous to \cite[Theorem~3]{Siciliano2025}, so we omit the proof.

\subsection{Functional reduction} 

Let $u\in H_{0}^{1}(\Omega)$ be fixed. 
We consider the linear functional 
\[
L_{u} : \mathbb H \longrightarrow \R, \qquad
L_{u}(\xi) := \int_{\Omega} q(x)\,u^{2}\,\xi \,dx.
\]
By H\"older's inequality and Sobolev embedding $H_{0}^{1}(\Omega)\hookrightarrow L^{4}(\Omega)$, we obtain for every $\xi\in\mathbb H$,
\begin{equation}\label{2.8}
	\big|L_{u}(\xi)\big|
	\le |q(x)|_{\infty} |u|_{4}^{2} |\xi|_{2}
	\le C_{1} |q(x)|_{\infty} |u|_{4}^{2} \|\xi\|_{\mathbb H}
	\le C_{2} \|u\|^{2} \|\xi\|_{\mathbb H},
\end{equation}
for some constant $C>0$ depending only on $\Omega$. 
Hence $L_{u}$ is continuous on $\mathbb H$. 
By the Riesz representation theorem there exists a unique element $\Phi(u)\in \mathbb H$ such that 
\begin{equation}\label{2.9}
	L_{u}(\xi) = (\Phi(u),\xi)_{\mathbb H}
	= \int_{\Omega} \nabla \Phi(u)\cdot \nabla \xi \,dx
	+ \int_{\Omega} \Delta \Phi(u)\,\Delta \xi \,dx,
	\qquad \forall\, \xi\in\mathbb H.
\end{equation}
In other words, for every $u\in H_{0}^{1}(\Omega)$ the Bopp--Podolsky equation 
\[
-\Delta \phi + \Delta^{2}\phi = q(x)u^{2} \quad \text{in }\Omega, 
\qquad \phi = \Delta\phi = 0 \quad \text{on }\partial\Omega,
\]
admits a unique weak solution $\phi=\Phi(u)\in\mathbb H$. 

Taking $\xi=\Phi(u)$ in \eqref{2.9} we obtain the identity 
\begin{equation}\label{2.10}
	(\Phi(u),\Phi(u))_{\mathbb H}
	= \int_{\Omega} q(x) u^{2} \Phi(u)\,dx,
\end{equation}
which will be used repeatedly in what follows. 

Using \eqref{2.10}, we obtain 
\begin{align*}
	\mathcal F(u,\Phi(u))
	&= \frac12 \int_{\Omega} |\nabla u|^{2} \,dx
	+ \frac12 \int_{\Omega} q(x)\,u^{2}\,\Phi(u)\,dx
	- \int_{\Omega} F(u)\,dx
	- \frac14 \|\Phi(u)\|_{\mathbb H}^{2} \\[2mm]
	&= \frac12 \int_{\Omega} |\nabla u|^{2} \,dx
	- \int_{\Omega} F(u)\,dx
	+ \frac14 \int_{\Omega} q(x)\,u^{2}\,\Phi(u)\,dx.
\end{align*}
We define
\begin{equation}\label{2.11}
	J(u) := \frac12 \int_{\Omega} |\nabla u|^{2} \,dx
	- \int_{\Omega} F(u)\,dx
	+ \frac14 \int_{\Omega} q(x)\,u^{2}\,\Phi(u)\,dx,
	\qquad u\in H_{0}^{1}(\Omega).
\end{equation}
By \cite[Proposition~2]{Siciliano2025} the map 
\[
\Phi : H_{0}^{1}(\Omega) \longrightarrow \mathbb H, 
\qquad
u \longmapsto \Phi(u),
\]
defined in \eqref{2.9} is of class $C^{1}$. 
Moreover, $\Phi(u) = \bigl(-\Delta + \Delta^{2}\bigr)^{-1}\bigl(q(x)\,u^{2}\bigr),$
where $\bigl(-\Delta + \Delta^{2}\bigr)^{-1} : \mathbb H' \to \mathbb H$ denotes the Riesz isomorphism associated with $(\cdot,\cdot)_{\mathbb H}$. 
As a consequence, the following proposition shows that the constrained critical points on $S(\mu)$ of the reduced functional $J$ in \eqref{2.11} provide normalized solutions to problem \eqref{1.1}. 

\begin{proposition}\label{prop:reduced-critical}
	Let $\mu>0$ and  $u\in S(\mu)$, then $u$ is a critical point of $J$ on $S(\mu)$ if and only if $(u,\omega,\Phi(u))$ is a weak solution of \eqref{1.1} under \eqref{1.2}–\eqref{1.3}, where $\omega\in\R$ is the Lagrange multiplier associated with the constraint.
\end{proposition}

\begin{proof}
	Let $u\in S(\mu)$ and let $\omega\in\R$ be the Lagrange multiplier associated with the constraint.
	Differentiating \eqref{2.11} in the direction $v\in H_{0}^{1}(\Omega)$ and using \eqref{2.9}  we obtain 
	\[
	\langle J'(u),v\rangle
	= \int_{\Omega} \nabla u\cdot \nabla v \,dx
	- \int_{\Omega} f(u)\,v \,dx
	+ \frac12 \int_{\Omega} q(x)\,\Phi(u)\,u\,v \,dx.
	\]
	Hence $u$ is a critical point of $J$ on $S(\mu)$ if and only if 
	\[
	\langle J'(u),v\rangle = \omega \int_{\Omega} u v \,dx
	\quad \text{for all } v\in H_{0}^{1}(\Omega),
	\]
	which is exactly \eqref{2.3} with $\phi=\Phi(u)$. 
	Conversely, if $(u,\omega,\Phi(u))$ satisfies \eqref{2.3}–\eqref{2.4}, then substituting $\phi=\Phi(u)$ into the energy we recover \eqref{2.11}, so $u$ is a constrained critical point of $J$. 
	This concludes the proof. 
\end{proof}

We recall that, as a consequence of \eqref{2.9} together with the estimate \eqref{2.8}, we have the continuity bound
\[
\int_{\Omega} q(x)u^{2}\Phi(u)\,dx \;\le\; C \|u\|^{2} \|\Phi(u)\|_{\mathbb H},
\]
and hence
\begin{equation}\label{2.13}
	\|\Phi(u)\|_{\mathbb H} \;\le\; C \|u\|^{2}, \qquad u\in H_{0}^{1}(\Omega),
\end{equation}
for some constant $C>0$ depending only on $\Omega$ and $|q(x)|_{\infty}$. 

To obtain the compactness of Cerami sequences ,  we provide a uniform bound for $\Phi(u)$ when $\mu>0$ is fixed.

\begin{lemma}\label{lem:Phi-mu}
	There exists a constant $C_{\Phi}=C(|q(x)|_{\infty},\Omega)>0$ such that for every $u\in S(\mu)$
	\[
	\|\Phi(u)\|_{\mathbb H} \le C_{\Phi}\,\mu .
	\]
\end{lemma}

\begin{proof}
	By \eqref{2.9}, taking $\xi=\Phi(u)$, we have 
	\begin{equation}\label{2.Phi-id}
		\int_{\Omega} q(x)u^{2}\Phi(u)\,dx
		= \int_{\Omega} |\Delta \Phi(u)|^{2}\,dx
		+ \int_{\Omega} |\nabla \Phi(u)|^{2}\,dx
		= \|\Phi(u)\|_{\mathbb H}^{2}.
	\end{equation}
	Since $q\in L^{\infty}(\Omega)$ and $\Omega$ is bounded, H\"older’s inequality yields 
	\[
	\int_{\Omega} q(x)u^{2}\Phi(u)\,dx
	\le |q(x)|_{\infty} \int_{\Omega} |u|^{2} |\Phi(u)| \,dx
	\le |q(x)|_{\infty} \, |u|_{2}^{2} \, |\Phi(u)|_{\infty}.
	\]
	Since $u\in S(\mu)$, we have $|u|_{2}^{2}=\mu$, and hence 
	\begin{equation}\label{2.bound-1}
	\|\Phi(u)\|_{\mathbb H}^{2} \le |q(x)|_{\infty} \,\mu \, |\Phi(u)|_{\infty}.
	\end{equation}
	
	When $N=3$ (and, more generally, when $N\le 3$), the embedding $H^{2}(\Omega)\hookrightarrow L^{\infty}(\Omega)$ provides a constant $C_{1}=C_{1}(\Omega)>0$ such that 
	\[
	|\Phi(u)|_{\infty} \le C_{1} \,\|\Phi(u)\|_{H^{2}(\Omega)}.
	\]
	Moreover, by the boundary conditions (I) there exists $C_{2}=C_{2}(M,\Omega)>0$ such that 
	\[
	\|\Phi(u)\|_{H^{2}(\Omega)}^{2} \le C_{2} \, \|\Phi(u)\|_{\mathbb H}^{2}.
	\]
	Combining the last two estimates we obtain 
	\[
	|\Phi(u)|_{\infty}
	\le C_{1} C_{2}^{1/2} \,\|\Phi(u)\|_{\mathbb H}.
	\]
	By \eqref{2.bound-1} we obtain 
	\[
\|\Phi(u)\|_{\mathbb H}^{2} \le |q(x)|_{\infty} \,\mu \, C_{1} C_{2}^{1/2} \, \|\Phi(u)\|_{\mathbb H}.
	\]
	In conclusion, let $	C_{\Phi} := |q(x)|_{\infty} \, C_{1} C_{2}^{1/2} > 0$, we have

	\[
\|\Phi(u)\|_{\mathbb H} \le C_{\Phi}\,\mu .
\]
	This completes the proof. 
\end{proof}

Recall from Theorem~\ref{thm:2.1} that the problem reduces to finding critical points of 
$\mathcal{F}$ on $S(\mu)\times\mathbb{H}$. 
The next result shows the equivalence with the constrained problem for $J$ on $S(\mu)$.

\begin{proposition}\label{pr2.5}
	Let $(u,\phi)\in S(\mu)\times\mathbb H$ and $\omega\in\R$. The following assertions are equivalent:
	\begin{itemize}
		\item[(i)] $(u,\phi)$ is a critical point of $\mathcal F$ constrained to $S(\mu)\times\mathbb H$ and $\omega$ is the associated Lagrange multiplier; 
		\item[(ii)] $u$ is a critical point of $J$ constrained to $S(\mu)$, with the same Lagrange multiplier $\omega$, and $\phi=\Phi(u)$. 
	\end{itemize}
\end{proposition}

\begin{proof}
	Assume \emph{(i)}. Then, 
	\[
	\partial_{u}\mathcal F(u,\phi)=\omega\,u
	\quad\text{and}\quad
	\partial_{\phi}\mathcal F(u,\phi)=0.
	\]
	By Proposition~\ref{prop:reduced-critical} we  have $\phi=\Phi(u)$, then $J'(u)=\omega\,u,$
	which is exactly \emph{(ii)}.
	
	Conversely, suppose \emph{(ii)} holds. Then $J'(u)=\omega\,u$ and, by definition of $J$, we have $\phi=\Phi(u)$, which solves $\partial_{\phi}\mathcal F(u,\phi)=0$. 
	Hence
	\[
	\partial_{u}\mathcal F\bigl(u,\Phi(u)\bigr)
	= J'(u)
	= \omega\,u,
	\]
	so $(u,\Phi(u))$ is a critical point of $\mathcal F$ on $S(\mu)\times\mathbb H$ with multiplier $\omega$, that is, \emph{(i)} holds. 
	The proof is complete.
\end{proof}

\section{The existence and multiplicity results}\label{3}

In this section we prove Theorems~\ref{thm:1.1}--\ref{thm:1.3}. 
As explained in Section~\ref{2}, after eliminating the electrostatic potential $\phi$ through the linear Bopp--Podolsky problem, 
the analysis reduces to the study of the reduced functional $J$ on $S(\mu)$.

\subsection{The perturbation functional}

We first recall the Gagliardo–Nirenberg inequality. 
Since $\Omega\subset\R^{3}$ is bounded, for every $p\in[2,6)$ there exists a constant $C_{p}>0$, depending only on $p$ and on $\Omega$, such that 
\begin{equation}\label{3.1}
	|u|_{p}^{p} \le C_{p}\, |u|_{2}^{(1-\beta_{p})p} \, \|\nabla u\|_{2}^{\beta_{p}p}
	\qquad \forall\,u\in H_{0}^{1}(\Omega),
\end{equation}
where $\beta_{p} := \frac{3}{2} \Bigl(\frac1{2}-\frac1{p}\Bigr)\in[0,1)$, see \cite{FiorenzaFormicaRoskovecSoudsky2021}.

For $0<\mu<1$ we define
\[
U_{\mu} := \bigl\{\,u\in H_{0}^{1}(\Omega): |u|_{2}^{2} < \mu \,\bigr\},
\]
and  for $u\in U_{\mu}$, 
\begin{equation}\label{3.2}
	J_{r,\mu}(u) := J(u) - H_{r,\mu}(u),
\end{equation}
where the penalization term $H_{r,\mu}$ is given by 
\begin{equation}\label{3.3}
	H_{r,\mu}(u) := f_{r}\!\left(\frac{|u|_{2}^{2}}{\mu}\right), 
	\qquad
	f_{r}(s) := \frac{s^{r}}{1-s}, \quad s\in[0,1),
\end{equation}
and $r>1$ is a parameter to be fixed later. 
A direct computation shows that 
\begin{equation}\label{3.4}
	f_{r}'(s) = \frac{r s^{r-1}}{1-s} + \frac{s^{r}}{(1-s)^{2}} > 0, 
	\qquad s\in(0,1),
\end{equation}
and
\begin{equation}\label{3.5}
	f_{r}''(s) = \frac{r(r-1) s^{r-2}}{1-s}
	+ \frac{r s^{r-1}}{(1-s)^{2}}
	+ \frac{2 s^{r}}{(1-s)^{3}} > 0,
	\qquad s\in(0,1),
\end{equation}
hence $f_{r}$ is increasing and convex on $[0,1)$. 
In particular, the map $u \mapsto H_{r,\mu}(u)$ is of class $C^{1}$ on $U_{\mu}$. 
Since $J$ in \eqref{2.11} is $C^{1}$ on $H_{0}^{1}(\Omega)$, it follows that
\[
J_{r,\mu}=J-H_{r,\mu} \in C^{1}(U_{\mu},\mathbb R).
\] 
Recalling the expression of $J$ in \eqref{2.11}, we obtain for every $u\in U_{\mu}$ and $v\in H_{0}^{1}(\Omega)$ 
\begin{equation}\label{3.6}
	\bigl\langle J_{r,\mu}'(u), v \bigr\rangle
	= \int_{\Omega} \nabla u \cdot \nabla v \,dx
	+ \int_{\Omega} q(x)\,\Phi(u)\,u\,v \,dx
	- \int_{\Omega} f(u)\,v \,dx
	- \frac{2}{\mu} f_{r}'\!\left(\frac{|u|_{2}^{2}}{\mu}\right)
	\int_{\Omega} u\,v \,dx.
\end{equation}

Motivated by the argument in Lemma~2.5 of  \cite{Esteban-Sere1999}, and by its adaptation to the normalized setting in  \cite{Alves-He-Ji2025},  we now modify $J_{r,\mu}$ in order to work on the whole space $H_{0}^{1}(\Omega)$.

Let $\beta \in C^{\infty}(\mathbb R,\mathbb R)$ be such that
\[
\begin{cases}
	\beta \equiv -1 & \text{on } (-\infty,-1),\\[2pt]
	\beta(t) = t     & \forall\, t \ge 0,\\[2pt]
	\beta(t) \le 0   & \forall\, t \le 0.
\end{cases}
\]
We define
\begin{equation}\label{3.cutoff-functional}
	E_{r,\mu} : H_{0}^{1}(\Omega) \longrightarrow \R, \qquad
	E_{r,\mu}(u) :=
	\begin{cases}
		\beta\bigl(J_{r,\mu}(u)\bigr), & \text{if } |u|_{2}^{2} < \mu,\\[3pt]
		-1, & \text{otherwise.}
	\end{cases}
\end{equation}

By \cite[Lemma~7.1]{Alves-He-Ji2025}, the map $E_{r,\mu}$ belongs to $C^{1}(H_{0}^{1}(\Omega),\R)$. Moreover, if $u$ is a critical point of $E_{r,\mu}$ with $E_{r,\mu}(u)\ge0$, then automatically $|u|_{2}^{2}<\mu$, so $u\in U_{\mu}$ and, in fact, $u$ is a critical point of $J_{r,\mu}$ at the same level.
The same implication holds for Cerami sequences.
Therefore, instead of working directly with $J_{r,\mu}$, we may look for positive min–max levels of $E_{r,\mu}$.
This is more convenient because $E_{r,\mu}$ is globally defined on $H_{0}^{1}(\Omega)$ and satisfies $E_{r,\mu}=-1$ on $\partial U_{\mu}$.

\subsection{Compactness of Cerami sequences}

Fix $\mu>0$. For $r>1$ sufficiently large, let $\{u_{n,r}\}\subset H_0^1(\Omega)$ be a \emph{Cerami sequence} for $J_{r,\mu}$ at level $c_r>0$, i.e.,
\begin{equation}\label{3.Cerami}
	J_{r,\mu}(u_{n,r}) \to c_{r} \quad\text{and}\quad
	(1+\|u_{n,r}\|)\,\|J_{r,\mu}'(u_{n,r})\| \to 0 
	\qquad \text{as } n \to +\infty.
\end{equation}

 Assume moreover that the map $r \mapsto c_{r}$ is nondecreasing and set
\begin{equation}\label{3.lambda-nr}
	\lambda_{n,r} := \frac{2}{\mu}\, f_{r}'\!\left(\frac{|u_{n,r}|_{2}^{2}}{\mu}\right).
\end{equation}

\begin{lemma}\label{lem:3.1}
	For $r>1$ sufficiently large, the sequence $\{u_{n,r}\}_{n\ge1}$ given in \eqref{3.Cerami} is bounded in $H_{0}^{1}(\Omega)$. 
\end{lemma}

\begin{proof}
 Assume \emph{($f_1$)--($f_3$)} hold. 
	For every $n$, let $t_n\in[0,1]$ be such that
	\[
	J_{r,\mu}(t_n u_{n,r})=\max_{t\in[0,1]} J_{r,\mu}(t u_{n,r}).
	\]
	We claim that the sequence $\{J_{r,\mu}(t_n u_{n,r})\}_{n}$ is bounded from above. 
	Indeed, if $t_n=0$ or $t_n=1$ there is nothing to prove, hence we assume $t_n\in(0,1)$, and so,
	$\bigl\langle J_{r,\mu}'(t_n u_{n,r}),\, t_n u_{n,r} \bigr\rangle = 0.$	Hence,
	\begin{align*}
		2 J_{r,\mu}(t_n u_{n,r})
		&= 2 J_{r,\mu}(t_n u_{n,r}) - \bigl\langle J_{r,\mu}'(t_n u_{n,r}),\, t_n u_{n,r} \bigr\rangle \\
		&= \int_{\Omega} \bigl[f(t_n u_{n,r})\, t_n u_{n,r} - 2 F(t_n u_{n,r}) \bigr]\,dx
		+ 2 f_r\!\left(\frac{t_n^{2} |u_{n,r}|_{2}^{2}}{\mu}\right) 
		\frac{t_n^{2} |u_{n,r}|_{2}^{2}}{\mu} \\[3pt]
		&\quad - 2 f_r\!\left(\frac{t_n^{2} |u_{n,r}|_{2}^{2}}{\mu}\right)
		- \frac12 \|\Phi(u_{n,r})\| _{\mathbb{H}}^{2}.
	\end{align*}
	By \emph{($f_2$)}, the map $s\mapsto f(s)s-2F(s)$ is nondecreasing on $(0,+\infty)$ and nonincreasing on $(-\infty,0)$, while $s\mapsto f_r'(s)s-f_r(s)$ is increasing on $(0,1)$. 
	Consequently,
	\begin{align*}
		2 J_{r,\mu}(t_n u_{n,r})
		&\le \int_{\Omega} \bigl[f(u_{n,r})\, u_{n,r} - 2 F(u_{n,r}) \bigr]\,dx
		+ 2 f_r\!\left(\frac{|u_{n,r}|_{2}^{2}}{\mu}\right) \frac{|u_{n,r}|_{2}^{2}}{\mu}
		- 2 f_r\!\left(\frac{|u_{n,r}|_{2}^{2}}{\mu}\right) \\
		&= 2 J_{r,\mu}(u_{n,r}) - \bigl\langle J_{r,\mu}'(u_{n,r}),\, u_{n,r} \bigr\rangle
		+ \frac12 \|\Phi(u_{n,r})\| _{\mathbb{H}}^{2} \\
		&\le 2 J_{r,\mu}(u_{n,r}) + o_n(1) + \frac12 C_{\Phi}^{2} \mu^{2}, 
	\end{align*}
where we used Lemma~\ref{lem:Phi-mu} and  \eqref{3.Cerami}. Hence $\{J_{r,\mu}(t_n u_{n,r})\}_{n}$ is bounded from above.
	
	We now show that $\{u_{n,r}\}_{n\ge1}$ is bounded. 
	Argue by contradiction, suppose
	$\|u_{n,r}\|\to +\infty \quad\text{as } n\to+\infty.$
	Set $v_n := \frac{u_{n,r}}{\|u_{n,r}\|},$
	so that $\|v_n\|=1$. Up to a subsequence, $v_n \rightharpoonup v$ in $H_0^1(\Omega)$ for some $v$. We show that $v=0$. 
	Otherwise, for
	\[
	\Omega_* := \{x\in\Omega : v(x)\neq 0\}
	\]
	we would have $|u_{n,r}(x)|\to +\infty$ for every $x\in\Omega_*$. Then, by \emph{($f_2$)}–\emph{($f_3$)} and Fatou’s lemma,
	\begin{align*}
		0
		= \lim_{n\to\infty} \frac{c_r}{\|u_{n,r}\|^{2}}
		= \lim_{n\to\infty} \frac{J_{r,\mu}(u_{n,r})}{\|u_{n,r}\|^{2}}
		&\le \lim_{n\to\infty} \left[ \frac12 \int_{\Omega_*} \frac{F(u_{n,r})}{|u_{n,r}|^{2}}\, v_n^{2}\,dx \right] + o_n(1) \\
		&\le \frac12 \liminf_{n\to\infty} \int_{\Omega_*} \frac{F(u_{n,r})}{|u_{n,r}|^{2}}\, v_n^{2}\,dx \\
		&\le \frac12 \int_{\Omega_*} \liminf_{n\to\infty} \frac{F(u_{n,r})}{|u_{n,r}|^{2}}\, v^{2}\,dx
		= -\infty, 
	\end{align*}
	which is a contradiction. Thus $v=0$. 
	
	Fix now $B>0$. For $n$ large enough we have $\frac{B}{\|u_{n,r}\|}\in[0,1]$. Since
	$\frac{B^{2} \|v_n\|_{2}^{2}}{\mu} < \frac{B^{2}}{\|u_{n,r}\|^{2}} \le 1,$
	we infer
	\[
	J_{r,\mu}(t_n u_{n,r})
	\ge J_{r,\mu}\!\left(\frac{B}{\|u_{n,r}\|} u_{n,r}\right)
	= J_{r,\mu}(B v_n)
	= \frac{B^{2}}{2} - \int_{\Omega} F(B v_n)\,dx
	- f_r\!\left(\frac{B^{2} \|v_n\|_{2}^{2}}{\mu}\right)
	+ \frac14 \|\Phi(B v_n)\| _{\mathbb{H}}^{2}.
	\]
	Since $H_0^1(\Omega)\hookrightarrow L^{p}(\Omega)$ is compact for $2\le p<6$, we obtain
	\[
	\liminf_{n\to\infty} J_{r,\mu}(t_n u_{n,r}) \ge \frac{B^{2}}{2}, \qquad \forall B>0,
	\]
	which contradicts the boundedness of $\{J_{r,\mu}(t_n u_{n,r})\}_{n}$. Hence $\{u_{n,r}\}_{n}$ is bounded in $H_0^1(\Omega)$ . 
	
	\medskip
	
Now assume \emph{($f_1$)} and \emph{($f_4$)} hold. 
	Taking $r>\frac{q}{2}$ in \eqref{3.3}, we obtain
	\begin{align*}
		2q\, J_{r,\mu}(u_{n,r}) - 2 \bigl\langle J_{r,\mu}'(u_{n,r}),\, u_{n,r} \bigr\rangle
		&= (q-2) \int_{\Omega} |\nabla u_{n,r}|^{2}\,dx
		+ \left(\frac{q}{2} - 2\right) \|\Phi(u_{n,r})\| _{\mathbb{H}}^{2} \\
		&\quad + 2 \left( \int_{\Omega} f(u_{n,r}) u_{n,r}\,dx
		- q \int_{\Omega} F(u_{n,r})\,dx \right) \\
		&\ge (q-2)\|u_{n,r}\|^{2}
		+ \left(\frac{q}{2} - 2\right) \|\Phi(u_{n,r})\| _{\mathbb{H}}^{2}. 
	\end{align*}
	By Lemma~\ref{lem:Phi-mu}, $\|\Phi(u_{n,r})\| _{\mathbb{H}}\le C_{\Phi}\mu$, and hence
	\[
	2q\, J_{r,\mu}(u_{n,r}) - 2 \bigl\langle J_{r,\mu}'(u_{n,r}),\, u_{n,r} \bigr\rangle
	\ge (q-2)\|u_{n,r}\|^{2}
	+ \left(\frac{q}{2} - 2\right) C_{\Phi}^{2} \mu^{2}. 
	\]
	Using the definition of $J_{r,\mu}$ and \eqref{3.Cerami} , we obtain
	\begin{equation}\label{3.10}
		\frac{2q\, J_{r,\mu}(u_{n,r}) - 2 \bigl\langle J_{r,\mu}'(u_{n,r}),\, u_{n,r} \bigr\rangle}{q-2}
		\le \frac{2q\, c_r}{q-2} + o_n(1),
	\end{equation}
	Consequently,
	\[
	\|u_{n,r}\|^{2}
	\le \frac{2q\, c_r}{q-2}
	+ \frac{2 - \frac{q}{2}}{q-2} C_{\Phi}^{2} \mu^{2}
	+ o_n(1).
	\]
	The proof is complete. 
\end{proof}

\begin{lemma}\label{lem:3.2}
	If $c_\infty := \lim_{r\to+\infty} c_r < +\infty,$
	then
	\[
	\lambda_r := \limsup_{n\to\infty} \frac{2}{\mu}\, f_r'\!\left(\frac{|u_{n,r}|_2^2}{\mu}\right) < +\infty.
	\]
	Moreover,
	\begin{equation}\label{3.7}
		\limsup_{r\to+\infty} \lambda_r
		\le \frac{2c_\infty}{\mu} + \frac{C_\Phi^{2}}{2}\,\mu.
	\end{equation}

\end{lemma}

\begin{proof}
	Since $r \mapsto c_r$ is non-decreasing and $c_{r_0}>0$ for some $r_0>1$ sufficiently large, we have $c_\infty \ge c_{r_0}>0$. By continuity of $s\mapsto f_r'(s)s-f_r(s)$, there exists $\xi_r\in(0,1)$ such that
	\begin{equation}\label{3.xir}
		c_\infty + \frac14 C_\Phi^2 \mu^2
		= f_r'(\xi_r)\,\xi_r - f_r(\xi_r).
	\end{equation}
	We claim that $\xi_r \to 1^{-}$ as $r\to+\infty$. 
	Otherwise, up to a subsequence, $\xi_r\to \xi\in[0,1)$ and, using the explicit expression of $f_r$, we would get 
\[
f_r'(\xi_r)\,\xi_r - f_r(\xi_r) \longrightarrow 0 
\quad \text{as } r \to +\infty.
\]
	which contradicts \eqref{3.xir}. Hence 
	\begin{equation}\label{3.xir-limit}
		\xi_r \to 1^{-}, \qquad
		f_r(\xi_r)\to 0, \qquad
		f_r'(\xi_r)\to c_\infty + \frac14 C_\Phi^2 \mu^2
		\quad\text{as } r\to+\infty.
	\end{equation}
	Since $f(t)t - 2F(t) \ge 0$ for all $t \in \mathbb{R}$, we obtain
	\begin{align*}
		2 J_{r,\mu}(u_{n,r}) + o_n(1)
		&= 2 J_{r,\mu}(u_{n,r}) - \langle J_{r,\mu}'(u_{n,r}), u_{n,r} \rangle \\
		&= \int_\Omega \bigl[f(u_{n,r})u_{n,r} - 2F(u_{n,r})\bigr]\,dx
		+ 2 f_r'\!\left(\frac{|u_{n,r}|_2^2}{\mu}\right)\frac{|u_{n,r}|_2^2}{\mu}
		- 2 f_r\!\left(\frac{|u_{n,r}|_2^2}{\mu}\right)
		- \frac12 \|\Phi(u_{n,r})\|_{\mathbb{H}}^2 \\
		&\ge 2 f_r'\!\left(\frac{|u_{n,r}|_2^2}{\mu}\right)\frac{|u_{n,r}|_2^2}{\mu}
		- 2 f_r\!\left(\frac{|u_{n,r}|_2^2}{\mu}\right)
		- \frac12 C_\Phi^2 \mu^2, 
	\end{align*}
Taking the upper limit as $n \to +\infty$, we obtain
	\begin{equation}\label{3.limsup}
		\limsup_{n\to\infty}
		\Biggl[
		f_r'\!\left(\frac{|u_{n,r}|_2^2}{\mu}\right)\frac{|u_{n,r}|_2^2}{\mu}
		- f_r\!\left(\frac{|u_{n,r}|_2^2}{\mu}\right)
		\Biggr]
		\le c_r + \frac14 C_\Phi^2 \mu^2
		\le c_\infty + \frac14 C_\Phi^2 \mu^2. 
	\end{equation}
	Since both $s \longmapsto f_r'(s)
	\quad\text{and}\quad
	s \longmapsto f_r'(s)s - f_r(s)	$
	are strictly increasing on $[0,1)$, comparison of \eqref{3.limsup} with \eqref{3.xir} yields 
	\[
	\limsup_{n\to\infty}
	f_r'\!\left(\frac{|u_{n,r}|_2^2}{\mu}\right)
	\le f_r'(\xi_r).
	\]
	Thus
	\[
	\lambda_r
	= \limsup_{n\to\infty} \frac{2}{\mu}\, f_r'\!\left(\frac{|u_{n,r}|_2^2}{\mu}\right)
	\le \frac{2}{\mu} f_r'(\xi_r),
	\]
	and letting $r\to+\infty$ and using \eqref{3.xir-limit} we derive \eqref{3.7}. 
	This completes the proof.
\end{proof}

 \begin{lemma}\label{le3.3}
 	If $c_{\infty}:=\lim_{r\to+\infty} c_r<+\infty,$
 	then for all $r>1$ sufficiently large, there exists $u_r\in U_\mu$ such that, up to a subsequence, $u_{n,r}\to u_r$ in $H_0^1(\Omega)$ as $n\to+\infty$. Moreover, 
 	\[
 	J_{r,\mu}(u_r)=c_r \quad \text{and} \quad J_{r,\mu}'(u_r)=0,
 	\]
 	with
 	\[
 	\frac{2}{\mu}\,f_r'\!\left(\frac{|u_r|_2^2}{\mu}\right)=\lambda_r
 	\quad\text{and}\quad
 	\limsup_{r\to+\infty}\lambda_r<\infty.
 	\]
 \end{lemma}
 
 \begin{proof}
 	Fix $r>1$ sufficiently large.
 	By Lemma~\ref{lem:3.2}, up to a subsequence, we may assume that $\lambda_{n,r}\to\lambda_r$ as $n\to+\infty$. 
 	By Lemma~\ref{lem:3.1}, $\{u_{n,r}\}_{n\ge1}$ is bounded in $H_0^1(\Omega)$; hence, up to a subsequence,
 	\[
 	u_{n,r}\rightharpoonup u_r \quad \text{in } H_0^1(\Omega)
 	\quad\text{and}\quad
 	u_{n,r}\to u_r \quad \text{in } L^2(\Omega).
 	\]
 	Since $\{u_{n,r}\}\subset U_\mu$, Fatou’s lemma yields that $|u_r|_2^2\le\mu$. We claim that $|u_r|_2^2<\mu$. Indeed, if $|u_r|_2^2=\mu$, then $|u_{n,r}|_2^2\to\mu$, and since $f_r(s)\to+\infty$ as $s\to1^{-}$, we would have
 	\[
 	J_{r,\mu}(u_{n,r})
 	= J(u_{n,r}) - f_r\!\left(\frac{|u_{n,r}|_2^2}{\mu}\right)
 	\longrightarrow -\infty
 	\quad \text{as } n \to +\infty.
 	\]
 	which contradicts $J_{r,\mu}(u_{n,r}) \to c_r > -\infty \quad \text{as } n \to +\infty.$
 	Hence $u_r\in U_\mu$.
 	
 	By~\cite[Theorem~10.1]{AgmonDouglisNirenberg1964}, the operator 
 	\[
 	\bigl(-\Delta+\Delta^2\bigr)^{-1}:L^p(\Omega)\longrightarrow W^{4,p}(\Omega),
 	\quad 1<p<\infty,
 	\]
 	is continuous. Hence
 \[
 \left\|\bigl(-\Delta+\Delta^2\bigr)^{-1}\!\left(u_{n,r}^2-u_r^2\right)\right\|_{H^2(\Omega)}
 \le C\,|u_{n,r}^2-u_r^2|_{2}\longrightarrow 0
 \quad \text{as } n \to +\infty.
 \]
 	Since $H^2(\Omega)\hookrightarrow L^{\infty}(\Omega)$ for $N\le3$, we obtain 
 	\[
 	|\Phi(u_{n,r})-\Phi(u_r)|_{\infty}\longrightarrow 0.
 	\]
 	Therefore,
 	\[
 	|\Phi(u_{n,r})u_{n,r}-\Phi(u_r)u_r|_{2}
 	\le |\Phi(u_{n,r})-\Phi(u_r)|_{\infty}\,|u_{n,r}|_2
 	+|\Phi(u_r)|_{\infty}\,|u_{n,r}-u_r|_2.
 	\]
 	Since the right-hand side tends to zero, we conclude that 
 	\[
 	|\Phi(u_{n,r})u_{n,r} - \Phi(u_r)u_r|_{2} \to 0.
 	\]
 	Let $\varphi\in H_0^1(\Omega)$. Using $\langle J_{r,\mu}'(u_{n,r}),\varphi\rangle=o_n(1)$ and the convergences above, we obtain 
 	\[
 	\begin{aligned}
 		\left\langle J'(u_r),\varphi\right\rangle-\lambda_r\!\int_{\Omega}u_r\varphi\,dx
 		&=\left\langle J_{r,\mu}'(u_{n,r}),\varphi\right\rangle
 		+\int_{\Omega}\nabla(u_r-u_{n,r})\cdot\nabla\varphi\,dx\\
 		&\quad+\int_{\Omega}\bigl(f(u_{n,r})u_{n,r}-f(u_r)u_r\bigr)\varphi\,dx\\
 		&\quad+\int_{\Omega}q(x)\,\varphi\bigl(\Phi(u_r)u_r-\Phi(u_{n,r})u_{n,r}\bigr)\,dx
 		+(\lambda_{n,r}-\lambda_r)\!\int_{\Omega}u_{n,r}\varphi\,dx\\
 		&=o_n(1).
 	\end{aligned}
 	\]
 	Hence $\langle J'(u_r),\cdot\rangle-\lambda_r(u_r,\cdot)_2=0,$
 	i.e. $J'(u_r)=\lambda_r u_r$ in $H^{-1}(\Omega)$. 
 	Next, taking $\varphi=u_{n,r}-u_r$, we have
 \[
 \int_{\Omega}\!\bigl(f(u_{n,r})u_{n,r}-f(u_r)u_r\bigr)(u_{n,r}-u_r)\,dx
 \longrightarrow 0
 \quad \text{as } n \to +\infty.
 \]
 	Combining this with $\langle J_{r,\mu}'(u_{n,r}),u_{n,r}-u_r\rangle=o_n(1)$ and using the compact embedding $H_0^1(\Omega)\hookrightarrow L^2(\Omega)$, we obtain 
 	\[
 	\|u_{n,r}-u_r\|^2
 	=\lambda_r\,|u_{n,r}-u_r|_2^2+o_n(1)=o_n(1),
 	\]
 	where $o_n(1)\to0$ as $n\to\infty$, and therefore $u_{n,r}\to u_r$ strongly in $H_0^1(\Omega)$.
 	
 	Finally, the strong convergence implies 
 	\[
 	J_{r,\mu}(u_r)=\lim_{n\to\infty}J_{r,\mu}(u_{n,r})=c_r,
 	\quad
 	J_{r,\mu}'(u_r)=\lim_{n\to\infty}J_{r,\mu}'(u_{n,r})=0, 
 	\]
 	and
 	\[
 	\lambda_r=\lim_{n\to\infty}\lambda_{n,r}
 	=\lim_{n\to\infty}\frac{2}{\mu}\,f_r'\!\left(\frac{|u_{n,r}|_2^2}{\mu}\right)
 	=\frac{2}{\mu}\,f_r'\!\left(\frac{|u_r|_2^2}{\mu}\right),
 	\]
 	where the uniform bound $\limsup_{r\to+\infty}\lambda_r<\infty$ follows from Lemma~\ref{lem:3.2}.
 	This completes the proof.
 \end{proof}

\begin{lemma}\label{lem:3.4}
	If $c_{\infty}:=\lim_{r\to+\infty} c_r<+\infty$, then there exists $u\in H_0^1(\Omega)$ such that, up to a subsequence, $u_{r_n}\to u$ in $H_0^1(\Omega)$ as $n\to+\infty$. Moreover, $J(u)=c_{\infty}$, and 
	\begin{itemize}
		\item[(i)] either $u$ is a critical point of $J$ constrained on $S(\mu)$ with Lagrange multiplier
		\[
		\lambda \in \left[\,0,\ \frac{2c_{\infty}}{\mu}+\frac{C_{\Phi}^2\mu}{2}\,\right]; 
		\]
		
		\item[(ii)] or $u$ is a critical point of $J$ constrained on $S(\nu)$ for some $0<\nu<\mu$, with Lagrange multiplier $\lambda=0$. 
	\end{itemize}
\end{lemma}

\begin{proof}
	By construction, let $(r_n)_{n\ge1}\subset(1,+\infty)$ with $r_n \nearrow +\infty$, $\{u_{r_n}\}_{n\ge1}$ satisfies 
	\[
	J_{r_n,\mu}(u_{r_n})=c_{r_n}
	\quad\text{and}\quad
	J_{r_n,\mu}'(u_{r_n})=0, 
	\]
	with
	\[
	\frac{2}{\mu}\,f_{r_n}'\!\left(\frac{|u_{r_n}|_2^2}{\mu}\right)=\lambda_{r_n}
	\quad\text{and}\quad
	\limsup_{n\to+\infty}\lambda_{r_n}<\infty.
	\]
	Repeating the arguments of Lemmas~\ref{lem:3.1} and~\ref{lem:3.2}, we may assume (up to a subsequence) that
	\[
	u_{r_n}\to u \text{ in } H_0^1(\Omega)
	\quad\text{and}\quad
	\lambda_{r_n}\to\lambda \text{ as } n\to+\infty. 
	\]
	By~\eqref{3.4}, we  have
	\[
	f_{r_n}\!\left(\frac{|u_{r_n}|_2^2}{\mu}\right)\to0 \quad \text{as } n\to+\infty.
	\]
	Hence 
	$J(u)=c_{\infty},\;
	\langle J'(u),\cdot\rangle-\lambda\,(u,\cdot)_2=0,\;
	|u|_2^2\le \mu,\;
	0\le \lambda<\infty.$

	Thus, either $|u|_2^2=\mu$ or $|u|_2^2<\mu$. In the latter case,
	\[
	0\le \lambda
	=\lim_{n\to+\infty}\frac{2}{\mu}\,f_{r_n}'\!\left(\frac{|u_{r_n}|_2^2}{\mu}\right)
	\le \limsup_{n\to+\infty}\frac{2}{\mu}\,f_{r_n}'\!\left(\frac{\mu+|u|_2^2}{2\mu}\right)
	=0, 
	\]
	where the last equality follows from the properties of $f_r'$ in~\eqref{3.3}--\eqref{3.5}. 
	Consequently, $\lambda=0$.
	
	The proof is complete.
\end{proof}

\subsection{A non-existence result.}
Next, we present a key lemma concerning a non-existence result, which will be used to rule out case (ii) in lemma~\ref{lem:3.4}.

\begin{lemma}\label{lem:3.5}
	The following two nonexistence results hold:
	\begin{enumerate}[label=\rm(\roman*)]
		\item Assume that \emph{($f_1$)}–\emph{($f_3$)} hold. Then, for any $M>0$, there exists a positive constant $\mu_0$ such that, for all $0<\nu<\mu_0$, there is no $u\in H_0^1(\Omega)$ satisfying
		\[
		J'(u)=0,\quad J(u)\le M,\quad \text{and}\quad \int_{\Omega}u^2\,dx=\nu.
		\]
		
		\item Assume that \emph{($f_1$)} and \emph{($f_4$)} hold. Then, for any $M>0$ sufficiently large, there exists a positive constant $\mu^*>0$ such that, for all $0<\nu<\mu^*$, there is no $u\in H_0^1(\Omega)$ satisfying
		\[
		J'(u)=0,\quad J(u)\le M\mu^*,\quad \text{and}\quad \int_{\Omega}u^2\,dx=\nu.
		\]
	
	\end{enumerate}
\end{lemma}

\begin{proof}
	(i) Suppose that \emph{($f_1$)}–\emph{($f_3$)} hold. For any fixed $M>0$ we claim that there exists $R>0$, depending only on $\Omega$, $f$, $q(x)$ and $M$, such that for all $u\in H_0^1(\Omega)$ with $J'(u)=0$ and $J(u)\le M$ one has $\|u\|\le R$. 
	Indeed, if not, then there exists a sequence $\{u_n\}\subset H_0^1(\Omega)$ satisfying $J'(u_n)=0$, $J(u_n)\le M$, and $\|u_n\|\to+\infty$ as $n\to+\infty$. Repeating the argument of Lemma~\ref{lem:3.1} leads to a contradiction. 
	
	Now argue by contradiction and assume that for some $\nu>0$ there exists $u\in H_0^1(\Omega)$ such that $J'(u)=0$, $J(u)\le M$ and $\int_{\Omega}u^2\,dx=\nu$. By $J'(u)=0$, \emph{(f$_1$)} and~\eqref{3.1} we obtain
	\[
	\begin{aligned}
		\|u\|^2
		&\le \int_{\Omega} f(u)u\,dx
		\le K_2\,\|u\|_2^2+K_p\,C_{p}\,\|u\|^{\beta_p p}\,\|u\|_2^{(1-\beta_p)p}\\
		&\le \frac{K_2}{\lambda_{1}}\|u\|^2+K_p\,C_{p}\,\|u\|^{\beta_p p}\,\|u\|_2^{(1-\beta_p)p}.
	\end{aligned}
	\]
	Hence,
	\[
	\frac{\lambda_{1}-K_2}{\lambda_{1}}
	\le K_p\,C_{p}\,\|u\|^{\beta_p p-2}\,\nu^{\frac{(1-\beta_p)p}{2}}.
	\]
	If $2<p\le \frac{10}{3}$, then $\beta_p p\le2$, and thus
	\[
	\frac{\lambda_{1}-K_2}{\lambda_{1}}
	\le K_p\,C_{p}\,\lambda_{1}^{\frac{\beta_p p-2}{2}}\,\nu^{\frac{p-2}{2}},
	\]
	and if $\frac{10}{3}<p<6$, then $\beta_p p>2$, we have
	\[
	\frac{\lambda_{1}-K_2}{\lambda_{1}}
	\le K_p\,C_{p}\,R^{\beta_p p-2}\,\nu^{\frac{(1-\beta_p)p}{2}}.
	\]
	Therefore there exists $\mu_0>0$ such that, for any $0<\nu<\mu_0$, no $u\in H_0^1(\Omega)$ satisfies the above conditions. 
	
	\medskip
	
	(ii) Suppose that \emph{($f_1$)} and \emph{($f_4$)} hold. Let
	\[
	A=\frac{\lambda_{1}-K_2}{\lambda_{1}},\quad
	a=\frac{2qM}{q-2},\quad
	b=\frac{\bigl(2-\tfrac{q}{2}\bigr)C_{\Phi}^2\mu^2}{q-2},\quad
	\alpha=\frac{\beta_p p-2}{2},\quad
	\gamma=\frac{(1-\beta_p)p}{2}. 
	\]
	Define
	\[
	k(x)=K_p\,C_{p}\,(a x+b)^{\alpha}\,x^{\gamma}.
	\]
	A direct computation shows that $k$ attains its minimum at $x=\dfrac{-\gamma b}{a(\alpha+\gamma)}$, where 
	\[
	k_{\min}
	=K_p\,C_{p}\,b^{\alpha+\gamma}\!\left(\frac{2qM}{q-2}\right)^{-\gamma}
	\!\left(\frac{\alpha}{\alpha+\gamma}\right)^{\!\alpha}
	\!\left(\frac{-\gamma}{\alpha+\gamma}\right)^{\!\gamma}. 
	\]
	Since $k_{\min}$ decreases monotonically in $M$ and $k_{\min}\to0$ as $M\to\infty$, while $A>0$ is fixed, there exists $M_0>0$ such that $k_{\min}<A$ whenever $M>M_0$. 
	By continuity of $k$, and since the function is decreasing from $+\infty$ to $k_{\min}$ and then increasing from $k_{\min}$ to $+\infty$. It follows that when $M>M_0$ the equation
	\[
	\frac{\lambda_{1}-K_2}{\lambda_{1}}
	=K_p\,C_{p}\left(\frac{2qMx}{q-2}+\frac{\left(2-\tfrac{q}{2}\right)C_{\Phi}^2\mu^2}{q-2}\right)^{\frac{\beta_p p-2}{2}}x^{\frac{(1-\beta_p)p}{2}}
	\]
	admits a positive real solution, which we denote by $\mu^*$. 
	
	Now fix $M>M_0$ and argue by contradiction. Suppose that there exists $u\in H_0^1(\Omega)$ such that $J'(u)=0$, $J(u)\le M\mu^*$ and $\int_{\Omega}u^2\,dx=\nu$ for some $0<\nu<\mu^*$. Then
	\[
	\|u\|^2
	\le \frac{2q\,J(u)-2\langle J'(u),u\rangle}{q-2}
	\le \frac{2qM\mu^*}{q-2}+\frac{\bigl(2-\tfrac{q}{2}\bigr)}{q-2}C_{\Phi}^2\mu^2.
	\]
	Using~\eqref{3.1} and \emph{($f_1$)}, we have
	\[
	\|u\|^2
	= \int_{\Omega} f(u)\,u\,dx - \int_{\Omega}q(x)u^2\Phi(u)\,dx
	\le \frac{K_2}{\lambda_{1,0,0}}\|u\|^2+K_p\,C_{p}\,\|u\|^{\beta_p p}\,|u|_2^{(1-\beta_p)p}.
	\]
Repeating the previous argument, we derive that, if  $2<p\le \frac{10}{3}$
	\begin{equation}\label{3.13}
		\frac{\lambda_{1}-K_2}{\lambda_{1}}
		\le K_p\,C_{p}\,\|u\|^{\beta_p p-2}\,\nu^{\frac{(1-\beta_p)p}{2}},
	\end{equation}
	and if $\frac{10}{3}<p<6$, then
	\[
	\frac{\lambda_{1}-K_2}{\lambda_{1}}
	\le K_p\,C_{p}\!\left(\frac{2qM\nu}{q-2}+\frac{\bigl(2-\tfrac{q}{2}\bigr)C_{\Phi}^2\mu^2}{q-2}\right)^{\!\frac{\beta_p p-2}{2}}\nu^{\frac{(1-\beta_p)p}{2}}.
	\]
	Since $0<\nu<\mu^*$, we deduce that 
	\[
	\frac{\lambda_{1}-K_2}{\lambda_{1}}
	< K_p\,C_{p}\!\left(\frac{2qM\mu^*}{q-2}+\frac{\bigl(2-\tfrac{q}{2}\bigr)C_{\Phi}^2\mu^2}{q-2}\right)^{\!\frac{\beta_p p-2}{2}}(\mu^*)^{\frac{(1-\beta_p)p}{2}},
	\]
	which contradicts the definition of $\mu^*$. 
	Note that $\dfrac{2-\beta_p p}{p-2}=\dfrac{2}{p-2}-\dfrac{3}{2}$.
	Therefore both claims are proved and the lemma follows.
\end{proof}

\subsection{Proofs of Theorems~\ref{thm:1.1}--\ref{thm:1.3}}

We first deal with the existence result under Navier boundary conditions.
\begin{lemma}\label{lem:3.6}
	The functional $J_{r,\mu}$ satisfies the following properties:
	\begin{enumerate}[label=\rm(\roman*)]
		\item There exist $\alpha,\rho>0$ such that $J_{r,\mu}(u)\ge \alpha$ for all $u\in U_\mu$ with $\|u\|=\rho$. 
		\item There exists $e\in U_\mu$ with $\|e\|>\rho$ such that $J_{r,\mu}(e)<0$. 
	\end{enumerate}
\end{lemma}

\begin{proof}
	(i) Fix $\rho>0$ sufficiently small and take $u\in H_0^1(\Omega)$ with $\|u\|=\rho$. Then
	\[
	|u|_2^2 \le \frac{1}{\lambda_{1}}\|u\|^2 = \frac{\rho^2}{\lambda_{1}} < \mu.
	\]
	By the definition of $J_{r,\mu}$,
	\[
	\begin{aligned}
		J_{r,\mu}(u)
		&= \frac12 \int_{\Omega} |\nabla u|^2\,dx - \int_{\Omega} F(u)\,dx + \frac14 \|\Phi(u)\|^2
		- f_r\!\left(\frac{|u|_2^2}{\mu}\right)\\
		&\ge \frac12 \|u\|^2 - \int_{\Omega} F(u)\,dx - f_r\!\left(\frac{|u|_2^2}{\mu}\right), 
	\end{aligned}
	\]
	since $\frac14\|\Phi(u)\|_{\mathbb{H}}^2\ge0$. 
	From \eqref{3.1} and \emph{($f_1$)}
	we have
	\[
	J_{r,\mu}(u)
	\ge \rho^2\!\left(
	\frac12 - \frac{K_2}{2\lambda_{1}} - C \rho^{p-2}
	- \frac{(\mu\lambda_{1})^{-r} \rho^{2r-2}}{1 - (\mu\lambda_{1})^{-1}\rho^2}
	\right), 
	\]
	with $C>0$ depending only on $\Omega$, $p$, $N$.
	Since $K_2<\lambda_{1}$, $2<p<6$ and $r>1$,  there exists $\alpha>0$ such that $J_{r,\mu}(u)\ge\alpha$ whenever $u\in U_\mu$ and $\|u\|=\rho$. 
	
	(ii) Choose $u_0\in H_0^1(\Omega)$ such that $|u_0|_2^2=\mu$. Then for $t\in(0,1)$ we have $t u_0\in U_\mu$ and
	\[
	J_{r,\mu}(t u_0)
	= \frac{t^2}{2}\|u_0\|^2 - \int_{\Omega} F(tu_0)\,dx + \frac14 \|\Phi(tu_0)\|^2
	- f_r(t^2).
	\]
	Since $f_r(s)\to+\infty$ as $s\to1^{-}$, it follows that
	\[
	\lim_{t\to1^-} J_{r,\mu}(t u_0) = -\infty.
	\]
	Therefore there exists $t_0\in(0,1)$ such that $\|t_0 u_0\|>\rho$ and $J_{r,\mu}(t_0 u_0)<0$. 
	The proof is complete.
\end{proof}

By Theorem~7.2 in \cite{Alves-He-Ji2025}, the minimax value
\[
c_r := \inf_{\gamma \in \Gamma_{r,\mu}} \, \max_{t\in[0,1]} J_{r,\mu}(\gamma(t)) > 0,
\]
is well defined, where
\[
\Gamma_{r,\mu}
:= \bigl\{\, \gamma \in C([0,1],U_\mu) : \gamma(0)=0,\ J_{r,\mu}(\gamma(1))<0 \,\bigr\}.
\]

If $r_1 \le r_2$, we have $c_{r_1} \le c_{r_2}$ because $J_{r_1,\mu} \le J_{r_2,\mu}$ on $U_\mu$.
For any $r>1$, the mountain–pass geometry allows us to find a Cerami sequence $\{u_{n,r}\}_{n\ge1}$ satisfying
\[
J_{r,\mu}(u_{n,r}) \to c_r,
\qquad
\bigl(1+\|u_{n,r}\|\bigr)\,\|J_{r,\mu}'(u_{n,r})\| \to 0.
\]
Define$c_\infty = c_\infty(\mu) := \sup_{r>1} c_r = \lim_{r\to+\infty} c_r.$
For all $u \in H_0^1(\Omega)\setminus\{0\}$ with $|u|_2^2 = \mu$, we have
\begin{equation}\label{eq:c-infty-upper}
	c_\infty \;\le\; \sup_{0\le t<1} J(tu).
\end{equation}

\begin{lemma}\label{energy bound }
$c_\infty \le \frac{\mu \lambda_{1}}{2} + \frac14 C_{\Phi}^{2}\mu^{2}.$
\end{lemma}

\begin{proof}
	Let $\varphi_{1}$ be the first eigenfunction corresponding to the eigenvalue $\lambda_{1}$, normalized so that $|\varphi_{1}|_{2}^{2}=\mu$. Then
	\[
	\sup_{0\le t<1} J(t\varphi_{1})
	= \sup_{0\le t<1} \Bigl[\, \frac{t^{2}\mu\lambda_{1}}{2}
	- \int_{\Omega} F\bigl(t\varphi_{1}\bigr)\,dx
	+ \frac14 \|\Phi(t\varphi_{1})\|_{\mathbb{H}}^{2} \Bigr]
	\le \frac{\mu \lambda_{1}}{2} + \frac14 C_{\Phi}^{2}\mu^{2},
	\]
	where the inequality follows from the boundedness of $\Phi(t\varphi_{1})$ in $\mathbb H$ and from $|t|<1$. This completes the proof.
\end{proof}

We are now in the position to prove Theorem~\ref{thm:1.1}. 

\begin{proof}[Proof of Theorem~\ref{thm:1.1}]
	By Lemmas~\ref{lem:3.4} and~\ref{lem:3.6}, there exists $u\in H_0^1(\Omega)$ such that
	\[
	J(u)=c_{\infty}\le\frac{\mu\,\lambda_{1}}{2}+\frac14\,C_{\Phi}^{2}\mu^{2}, 
	\]
	and one of the following alternatives holds:
	\begin{itemize}
		\item[(i)] $u$ is a critical point of $J$ constrained to $S(\mu)$, with Lagrange multiplier $\lambda\in\R$; 
		\item[(ii)] $u$ is a critical point of $J$ constrained to $S(\nu)$ for some $0<\nu<\mu$, with Lagrange multiplier $\lambda=0$. 
	\end{itemize}
	
	We first assume that \emph{($f_1$)}–\emph{($f_3$)} hold. Let $\mu_0$ be defined as in Lemma~\ref{lem:3.5}, and set $\mu_0^*=\min\{1,\mu_0\}$. 
	If $0<\mu<\mu_0^*$, then
	\[
	J(u)<\frac{\mu_0^*\,\lambda_{1}}{2}+\frac14\,C_{\Phi}^{2}\mu^{2}
	\le\frac{\lambda_{1}}{2}+\frac14\,C_{\Phi}^{2}\mu^{2}. 
	\]
	Applying Lemma~\ref{lem:3.5} with $M=\tfrac{\lambda_{1}}{2}+\tfrac14\,C_{\Phi}^{2}\mu^{2}$, we conclude that alternative~(ii) cannot occur, and hence $u$ is a critical point of $J$ constrained on $S(\mu)$ with Lagrange multiplier $\lambda \in \bigl(0,\ \lambda_{1}+C_{\Phi}^{2}\mu\bigr]. $
	
	Now assume that \emph{($f_1$)} and \emph{($f_4$)} hold. 
	Apply Lemma~\ref{lem:3.5} with $M$ sufficiently large, and let $\mu^*>0$ be as given there. Then, for every $0<\mu\le\mu^*$, alternative~(ii) is excluded. Hence $|u|_2^2=\mu$, and $u$ is a constrained critical point of $J$ on $S(\mu)$. 
	
	By Proposition~\ref{pr2.5}, the pair $(u,\Phi(u))$ is a critical point of $\mathcal F$ on $S(\mu)\times\mathbb H$, with associated Lagrange multiplier $\omega=\lambda$. 
	Therefore, by Theorem~\ref{thm:2.1}, the triple $(u,\omega,\Phi(u))$ is a weak solution of~\eqref{1.1} under the Navier boundary conditions, with prescribed mass~$\mu$. 
	This completes the proof of Theorem~\ref{thm:1.1}. 
\end{proof}

We now prepare the variational framework  for the  Theorem~\ref{thm:1.3}. 

\noindent
By~\cite{Evans2010}, the operator $(-\Delta,H_0^1(\Omega))$ admits an increasing sequence of distinct eigenvalues
\[
0<\lambda_1<\lambda_2<\lambda_3<\cdots.
\]
For each $2\le j<+\infty$, let $\varphi_j$ be an eigenfunction corresponding to $\lambda_j$. 

For $j\ge2$ and $r>1$, define
\[
Y_j:=\operatorname{span}\{\,u\in H_0^1(\Omega): -\Delta u=\lambda u \text{ for some } \lambda\le\lambda_j\,\}, 
\]
and
\[
Z_j:=Y_j^{\perp}\oplus\operatorname{span}\{\varphi_j\},
\]
where the orthogonality is taken with respect to the inner product of $H_0^1(\Omega)$.
Define
\[
B_{r,j}:=\{\,u\in Y_j:\ \|u\|\le\rho_{r,j}\,\}, 
\qquad
N_{r,j}:=\{\,u\in Z_j:\ \|u\|=\xi_{r,j}\,\},
\]
where $\rho_{r,j}>\xi_{r,j}>0$. 

By the Fountain Theorem~\cite{Willem1996}, applied to the functional $J_{r,\mu}$ defined above, we obtain the following result.

\begin{lemma}\label{le3.10} 
	For $r>1$ and $j\ge2$, define
	\[
	c_{r,j}:=\inf_{\gamma\in\Gamma_{r,j}}\ \max_{u\in B_{r,j}} E_{r,\mu}(\gamma(u)), 
	\qquad
	\Gamma_{r,j}:=\bigl\{\gamma\in C(B_{r,j},H_0^1(\Omega)):\ 
	-\gamma(u)=\gamma(-u) \text{ for all } u\in B_{r,j},\
	\gamma|_{\partial B_{r,j}}=\mathrm{id}\bigr\}.
	\]
	If
	\[
	b_{r,j}:=\inf_{\substack{u\in Z_j\\ \|u\|=\xi_{r,j}}} E_{r,\mu}(u)>0> 
	a_{r,j}:=\max_{\substack{u\in Y_j\\ \|u\|=\rho_{r,j}}} E_{r,\mu}(u), 
	\]
	then $c_{r,j}\ge b_{r,j}>0$. Moreover, for the functional $J_{r,\mu}$, there exists a Palais--Smale sequence $\{u_{n,r,j}\}_{n\ge1}$ such that
	\[
	J_{r,\mu}(u_{n,r,j})\to c_{r,j}, 
	\qquad
	J_{r,\mu}'(u_{n,r,j})\to0 
	\quad\text{as } n\to+\infty.
	\]
\end{lemma}

In view of \cite{BartoloBenciFortunato1983,Cerami1978}, the existence of the Palais--Smale sequence $\{u_{n,r,j}\}_{n\ge1}$ in Lemma~\ref{le3.10} can be replaced by the existence of a Cerami sequence $\{u_{n,r,j}\}_{n\ge1}$ satisfying 
\[
J_{r,\mu}(u_{n,r,j})\to c_{r,j},
\qquad
(1+\|u_{n,r,j}\|)\,\|J'_{r,\mu}(u_{n,r,j})\|\to0
\quad\text{as } n\to+\infty.
\]
For $r>1$ and $j\ge2$, set $\rho_{r,j}=\sqrt{\mu\,\lambda_j}$.

\begin{lemma}\label{le3.11}
	For $r>1$ and $j\ge2$, 
	\[
	c_{r,j}\le \frac{\mu\,\lambda_j}{2} + \frac14\,C_{\Phi}^{2}\mu^{2},
	\qquad
	a_{r,j}=-1. 
	\]
	Moreover, for any $j\ge2$ and $\varepsilon>0$, there exist $\xi_{\varepsilon,j}>0$ and $r_{\varepsilon,j}>1$ such that, for all $r\ge r_{\varepsilon,j}$ and $\xi_{r,j}=\xi_{\varepsilon,j}$,
	\[
	\widetilde{b}_{r,j}:=
	\inf_{\substack{u\in Z_j\\ \|u\|=\xi_{r,j}}} J_{r,\mu}(u)
	\ge
	\frac{\mu(\lambda_j - K_2)}{2}
	-\frac{K_p C_{p}}{p}\,\mu^{p/2}\lambda_j^{\beta_p p/2}
	-\varepsilon.
	\]
\end{lemma}

\begin{proof}
	For all $u\in Y_j$ we have $\|u\|^{2} \le \lambda_j |u|_2^{2}$. Hence, if $u\in Y_j\cap U_{\mu}$, then $|u|_2^{2} \le \mu$ and
	\[
	J_{r,\mu}(u)
	= \frac12 \|u\|^{2} - \int_{\Omega} F(u)\,dx + \frac14 \|\Phi(u)\|_{\mathbb{H}}^{2} - f_r\!\left(\frac{|u|_2^{2}}{\mu}\right).
	\]
	Using $\|u\|^{2} \le \lambda_j |u|_2^{2} \le \mu \lambda_j$, Lemma~\ref{lem:Phi-mu} and \emph{(f$_1$)}, we obtain
	\[
	J_{r,\mu}(u)
	\le \frac{\mu\,\lambda_j}{2} + \frac14 C_{\Phi}^{2} \mu^{2}. 
	\]
	Therefore
	\[
	c_{r,j}
	\le \sup_{u\in Y_j\cap U_{\mu}} J_{r,\mu}(u)
	\le \frac{\mu\,\lambda_j}{2} + \frac14 C_{\Phi}^{2} \mu^{2}. 
	\]
	Since $\rho_{r,j} = \sqrt{\mu\,\lambda_j}$, every $u\in Y_j$ with $\|u\|=\rho_{r,j}$ satisfies
	\[
	|u|_2^{2} \ge \frac{\|u\|^{2}}{\lambda_j} = \mu. 
	\]
	Thus $u\notin U_{\mu}$, and by the definition of $E_{r,\mu}$ we have $E_{r,\mu}(u)=-1$. Hence
	\[
	a_{r,j} = \max_{\substack{u\in Y_j\\ \|u\|=\rho_{r,j}}} J_{r,\mu}(u) = -1. 
	\]
	Let $k>1$ with $k\nearrow +\infty$ and set
	\[
	\xi_{\varepsilon,j} := \sqrt{\frac{k-1}{k}\,\mu\,\lambda_j}.
	\]
	For any $u\in Z_j$ with $\|u\|=\xi_{\varepsilon,j}$, by Poincaré’s inequality on $Z_j$ we have
	\[
	|u|_2^{2} \le \frac{\|u\|^{2}}{\lambda_j}
	= \mu\,\frac{k-1}{k}. 
	\]
	Using \eqref{3.1} and \emph{(f$_1$)} we get
	\[
	\int_{\Omega} F(u)\,dx
	\le \frac{K_2}{2} |u|_2^{2}
	+ \frac{K_p C_{p}}{p} |u|_2^{(1-\beta_p)p} \|u\|^{\beta_p p}. 
	\]
	Since $|u|_2^{2} \le \mu \frac{k-1}{k}$ and $\|u\| = \xi_{\varepsilon,j}$, we obtain
	\[
	\int_{\Omega} F(u)\,dx
	\le \frac{K_2}{2}\,\mu\,\frac{k-1}{k}
	+ \frac{K_p C_{p}}{p}\,
	\mu^{p/2} \left(\frac{k-1}{k}\right)^{p/2}
	\lambda_j^{\beta_p p/2}. 
	\]
	Moreover,
	\[
	f_r\!\left(\frac{|u|_2^{2}}{\mu}\right)
	= f_r\!\left(\frac{k-1}{k}\right)
	= \frac{\bigl(\frac{k-1}{k}\bigr)^{r}}{1-\frac{k-1}{k}}
	= k \left(\frac{k-1}{k}\right)^{r}. 
	\]
Hence,
\[
\begin{aligned}
	E_{r,\mu}(u)= J_{r,\mu}(u) 
	&\ge \frac12 \|u\|^{2} - \int_{\Omega} F(u)\,dx - f_r\!\left(\frac{|u|_2^{2}}{\mu}\right) \\
	&\ge \frac{k-1}{k}\,\frac{\mu\,\lambda_j}{2}
	- \frac{k-1}{k}\,\frac{K_2 \mu}{2}
	- \left(\frac{k-1}{k}\right)^{p/2} \frac{K_p C_{p}}{p}\,\mu^{p/2} \lambda_j^{\beta_p p/2}
	- k \left(\frac{k-1}{k}\right)^{r}. 
\end{aligned}
\]
	Let $k>1$ be fixed and sufficiently large. Then there exists $r_{\varepsilon,j}>1$ such that,
	\[
	k\left(\frac{k-1}{k}\right)^{r} \le \varepsilon \qquad \text{for all } r \ge r_{\varepsilon,j}.
	\]
	With this choice, for all $r\ge r_{\varepsilon,j}$ and $\xi_{r,j}=\xi_{\varepsilon,j}$ we have
	\[
	\widetilde{b}_{r,j}
	:= \inf_{\substack{u\in Z_j\\ \|u\|=\xi_{r,j}}} E_{r,\mu}(u)
	\ge
	\frac{\mu(\lambda_j - K_2)}{2}
	- \frac{K_p C_{p,N}}{p}\,\mu^{p/2} \lambda_j^{\beta_p p/2}
	- \varepsilon. 
	\]
	This completes the proof.
\end{proof} 

\begin{proof}[Proof of Theorem~\ref{thm:1.3}]
	For any $j\ge2$, if $1<r_1\le r_2$ and $c_{r_1,j}>0$, then, for all $\gamma\in\Gamma_{r_1,j}$, we have
	\[
	\max_{u\in B_{r_1,j}} E_{r_1,\mu}\big(\gamma(u)\big) > 0, 
	\]
	and thus
	\[
	\max_{u\in B_{r_1,j}} E_{r_1,\mu}\big(\gamma(u)\big)
	= \max_{u\in B_{r_1,j}} J_{r_1,\mu}\big(\gamma(u)\big). 
	\]
	Since $\rho_{r,j}$, $B_{r,j}$, and $\Gamma_{r,j}$ are independent of $r>1$, and $J_{r_1,\mu}\le J_{r_2,\mu}$ on $H_0^1(\Omega)$, 
	\[
	c_{r_1,j}
	= \inf_{\gamma\in\Gamma_{r_1,j}}\, \max_{u\in B_{r_1,j}} E_{r_1,\mu}\big(\gamma(u)\big) 
	= \inf_{\gamma\in\Gamma_{r_1,j}}\, \max_{u\in B_{r_1,j}} J_{r_1,\mu}\big(\gamma(u)\big) 
	\le \inf_{\gamma\in\Gamma_{r_2,j}}\, \max_{u\in B_{r_2,j}} J_{r_2,\mu}\big(\gamma(u)\big). 
	\]
	Hence, for all $\gamma\in\Gamma_{r_2,j}$, we have
	\[
	\max_{u\in B_{r_2,j}} J_{r_2,\mu}\big(\gamma(u)\big) > 0. 
	\]
	Therefore,
	\[
	c_{r_1,j}
	\le \inf_{\gamma\in\Gamma_{r_2,j}}\, \max_{u\in B_{r_2,j}} J_{r_2,\mu}\big(\gamma(u)\big)
	= \inf_{\gamma\in\Gamma_{r_2,j}}\, \max_{u\in B_{r_2,j}} E_{r_2,\mu}\big(\gamma(u)\big) 
	= c_{r_2,j}.
	\]
	For $j\ge2$, if $c_{r_1,j}>0$ for some $r_1>1$, define
	\[
	c_{\infty,j}:=\lim_{r\to+\infty} c_{r,j},
	\qquad\text{and set}\qquad
	c_{\infty,1}:=c_{\infty}. 
	\]
	By the definition of $E_{r,\mu}$, if $\widetilde{b}_{r,j}>0$, then $b_{r,j}=\widetilde{b}_{r,j}>0$. 
	Let $\lambda_{k_1}=\lambda_{1}$. Since $\lambda_j \to +\infty$ as $j\to\infty$, we can choose integers
	$2 \le k_2 < k_3 < \cdots < k_m$ such that
	\[
	\lambda_{1} < \lambda_{k_2} - K_2 \le \lambda_{k_2}
	< \cdots <
	\lambda_{k_{m-1}} - K_2 \le \lambda_{k_{m-1}}
	< \lambda_{k_m} - K_2 \le \lambda_{k_m}.
	\]
	Since $C_{\Phi}\mu\to0$ as $\mu\to0$, there exists $\widetilde{\mu}_{k_m}>0$ such that, for all $0<\mu<\widetilde{\mu}_{k_m}$, 
	\[
	\begin{aligned}
		\frac{\mu\lambda_{1}}{2}+\frac{1}{4}C_{\Phi}^2\mu^{2}
		&<
		\frac{\mu(\lambda_{k_2}-K_2)}{2}
		-\frac{K_p C_{p}}{p}\,\mu^{p/2}\lambda_{k_2}^{\beta p/2}
		\\
		&\le
		\frac{\mu\lambda_{k_2}}{2}+\frac{1}{4}C_{\Phi}^2\mu^{2}
		\\
		&< \cdots <
		\frac{\mu(\lambda_{k_m}-K_2)}{2}
		-\frac{K_p C_{p}}{p}\,\mu^{p/2}\lambda_{k_m}^{\beta p/2}
		\\
		&\le
		\frac{\mu\lambda_{k_m}}{2}+\frac{1}{4}C_{\Phi}^2\mu^{2}.
	\end{aligned}
	\]
	Fix $0<\mu<\widetilde{\mu}_{k_m}$. By Lemma~\ref{le3.11}, for any $2\le i\le m$ and any $\varepsilon_{k_i}>0$ sufficiently small, there exist $\xi_{k_i}>0$ and $r_{k_i}>1$ such that, for all $r\ge r_{k_i}$ and $\xi_{r,k_i}=\xi_{k_i}$,
	\[
	\frac{\mu\lambda_{k_i}}{2}+\frac{1}{4}C_{\Phi}^2\mu^{2}
	\;\ge\;
	c_{r,k_i}
	\;\ge\;
	b_{r,k_i}
	\;=\;
	\widetilde{b}_{r,k_i}
	\;=\;
	\inf_{\substack{u\in Z_{k_i}\\ \|u\|=\xi_{r,k_i}}} J_{r,\mu}(u) 
	\;\ge\;
	\frac{\mu(\lambda_{k_i}-K_2)}{2}
	-\frac{K_p C_{p}}{p}\,\mu^{p/2}\lambda_{k_i}^{\beta p/2}
	-\varepsilon_{k_i}.
	\]
	Choose $\varepsilon_{k_i}>0$ sufficiently small such that, for all $2\le i\le m$,
	\[
	\frac{\mu\lambda_{k_{i-1}}}{2}+\frac{1}{4}C_{\Phi}^2\mu^{2}
	\;\le\;
	\frac{\mu(\lambda_{k_i}-K_2)}{2}
	-\frac{K_p C_{p}}{p}\,\mu^{p/2}\lambda_{k_i}^{\beta p/2}
	-\varepsilon_{k_i}
	\;\le\;
	\frac{\mu\lambda_{k_i}}{2}+\frac{1}{4}C_{\Phi}^2\mu^{2}.
	\]
	Let $r_0=\max\{\,r_{k_i}\,:\,2\le i\le m\,\}$ and set $\xi_{r,k_i}=\xi_{k_i}$. Thus, for all $r\ge r_0$, we conclude,
	\[
	0<c_{\infty,1}
	\;\le\; \frac{\mu\lambda_{k_1}}{2}+\frac{1}{4}C_{\Phi}^2\mu^{2}
	\;\le\; b_{r,k_2}
	\;\le\; c_{\infty,k_2}
	\;\le\; \frac{\mu\lambda_{k_2}}{2}+\frac{1}{4}C_{\Phi}^2\mu^{2}
	\;\le\; \cdots \;\le\;
	b_{r,k_m}
	\;\le\; c_{\infty,k_m}
	\;\le\; \frac{\mu\lambda_{k_m}}{2}+\frac{1}{4}C_{\Phi}^2\mu^{2}.
	\]
	
	Since $\lambda_{k_m}<+\infty$, by Lemmas~\ref{lem:3.4} and~\ref{lem:3.6}, we conclude that, for any $2\le i\le m$, there exists $u_{k_i}\in H_0^1(\Omega)$ such that
	\[
	J(u_{k_i})=c_{\infty,k_i}
	\quad\text{and}\quad
	c_{\infty,k_{i-1}}<c_{\infty,k_i}\le \frac{\mu\lambda_{k_i}}{2}+\frac{1}{4}C_{\Phi}^{2}\mu^{2},
	\]
	where $c_{\infty,k_i}$ is attained along a Cerami sequence for $J_{r,\mu}$, 
	and one of the following alternatives holds: 
	\begin{itemize}
		\item[(i)] $u_{k_i}$ is a critical point of $J$ constrained on $S(\mu)$ with Lagrange multiplier $\omega_{k_i}$; 
		\item[(ii)] $u_{k_i}$ is a critical point of $J$ constrained on $S(\nu)$ for some $0<\nu<\mu$, with Lagrange multiplier $\omega_{k_i}=0$. 
	\end{itemize}
	
	First, assume that \emph{($f_1$)}–\emph{($f_3$)} hold. Applying Lemma~\ref{lem:3.5}, for
	\[
	0<\mu<\mu^*_{m}:=\min\Bigl\{\,1,\ \mu_0\!\Bigl(\Omega,f,\tfrac{\lambda_{k_m}}{2}\Bigr),\ \widetilde{\mu}_{k_m}\Bigr\},
	\]
	the alternative~(ii) cannot occur for any $1\le i\le m$. Hence $J$ has at least $m$ critical points $u_{k_1},u_{k_2},\dots,u_{k_m}$ constrained on $S(\mu)$, and the corresponding Lagrange multipliers satisfy $\omega_{k_i}\in(0,\lambda_{k_i}+C_{\Phi}^{2}\mu]$. 
	
	Now assume that \emph{($f_1$)} and \emph{($f_4$)} hold. Applying Lemma~\ref{lem:3.5}, by fixing $\mu^*>0$ as in Lemma~\ref{lem:3.5} and taking
	\[
	0<\mu<\mu^*_{m}:=\min\Bigl\{\ \mu^*(K_2,K_p,\Omega,\tfrac{\lambda_{k_m}}{2},p,q),\ \widetilde{\mu}_{k_m}\Bigr\},
	\]
	we deduce that the alternative~(ii) is excluded for any $1\le i\le m$. Thus $J$ has at least $m$ critical points $u_{k_1},u_{k_2},\dots,u_{k_m}$ constrained on $S(\mu)$, where the associated Lagrange multipliers satisfy $\omega_{k_i}\in(0,\lambda_{k_i}+C_{\Phi}^{2}\mu]$. 
\end{proof}

\section{Existence of ground states}\label{4}

In this section we prove the existence of ground states under the Navier boundary conditions. By Theorem~\ref{thm:1.1}, problem~\eqref{1.1} admits at least one normalized solution.

\medskip
\noindent
Define
\[
\mathcal{S}(\mu):=\Big\{\,u\in S(\mu):\ \langle J'(u),\cdot\rangle=\lambda_u\,(u,\cdot)_2
\ \text{for some }\lambda_u\in\R \,\Big\}, 
\]

\begin{definition}\label{def:4.2}
	We say that $(u,\lambda_u)$ is a (normalized) ground state solution to problem~\eqref{1.1} if $u\in \mathcal{S}(\mu)$ and
	\[
	J(u)=\inf_{w\in \mathcal{S}(\mu)} J(w). 
	\]
\end{definition}

\begin{proposition}[Poho\v{z}aev-type identity]\label{prop:Pohozaev}
	Let $\Omega\subset\R^{3}$ be a bounded $C^{2}$ domain, and let $u\in H^{2}(\Omega)\cap H_{0}^{1}(\Omega)$ solve
	\begin{equation}\label{eq:PDE-u}
		-\Delta u
		= f(u) - q(x)\,\Phi(u)\,u + \lambda u
		\quad \text{in }\Omega,\qquad
		u=0 \text{ on }\partial\Omega,
	\end{equation}
	where $\lambda\in\R$, $q\in C^{1}(\overline\Omega)$.
	Then the following Poho\v{z}aev-type identity holds:
	\begin{equation}\label{eq:Pohozaev}
		\begin{aligned}
			\frac12\int_{\partial\Omega} &(x\!\cdot\!\nu)\,|\nabla u|^{2}\,dS
			+ \frac14\int_{\partial\Omega} (x\!\cdot\!\nu)
			\big(|\nabla \Phi(u)|^{2}+(\Delta \Phi(u))^{2}\big)\,dS \\
			&= -\frac12\int_{\Omega} |\nabla u|^{2}\,dx
			+ 3\int_{\Omega} F(u)\,dx
			+ \frac32\,\lambda\int_{\Omega} u^{2}\,dx \\
			&\quad - \frac54\int_{\Omega} |\nabla \Phi(u)|^{2}\,dx
			- \frac74\int_{\Omega} (\Delta \Phi(u))^{2}\,dx
			- \frac12\int_{\Omega} (x\!\cdot\!\nabla q)\,u^{2} \Phi(u)\,dx .
		\end{aligned}
	\end{equation}
	In particular, if $\Omega$ is star-shaped with respect to the origin, then $x\!\cdot\!\nu\ge0$ on $\partial\Omega$, and the left-hand side of \eqref{eq:Pohozaev} is nonnegative. 
\end{proposition}

\begin{proof}
Multiplying \eqref{eq:PDE-u} by $x\!\cdot\!\nabla u$ and integrating over $\Omega$, we obtain
	\[
	0 = \int_{\Omega} \big(\Delta u + f(u) - q(x)\Phi(u)u + \lambda u\big)\,(x\!\cdot\!\nabla u)\,dx
	=: I_{\Delta} + I_{p} + I_{\Phi} + I_{\lambda}.
	\]
	By integration by parts,
	\begin{align}
		I_{\Delta}
		&= \frac12\int_{\Omega} |\nabla u|^{2}\,dx
		+ \frac12\int_{\partial\Omega} (x\!\cdot\!\nu)\,|\nabla u|^{2}\,dS, \label{eq:IDelta} 
		\\
		I_{p}
		&= - 3\int_{\Omega} F(u)\,dx, \label{eq:Ip}
		\\
		I_{\lambda}
		&= - \frac32\,\lambda \int_{\Omega} u^{2}\,dx. \label{eq:Ilambda}
	\end{align}
	For the nonlocal term,
\begin{equation}\label{eq:Iphi}
	\begin{aligned}
		I_{\Phi}
		&= - \int_{\Omega} q(x)\,\Phi(u)\,u\,(x\!\cdot\!\nabla u)\,dx \\
		&= \frac32 \int_{\Omega} \bigl(|\nabla\Phi(u)|^{2}+(\Delta\Phi(u))^{2}\bigr)\,dx
		+ \frac12 \int_{\Omega} (x\!\cdot\!\nabla q)\,u^{2} \Phi(u)\,dx
		+ \frac12 \int_{\Omega} q(x)\,(x\!\cdot\!\nabla \Phi(u))\,u^{2}\,dx .
	\end{aligned}
\end{equation}
	Since $(-\Delta+\Delta^{2})\Phi(u)=q(x)u^{2}$, we have
	$$
		\int_{\Omega} q(x)\,(x\!\cdot\!\nabla \Phi(u))\,u^{2}\,dx
		= - \frac12 \int_{\Omega} |\nabla \Phi(u)|^{2}\,dx
		+ \frac12 \int_{\Omega} (\Delta \Phi(u))^{2}\,dx
		+ \frac12 \int_{\partial\Omega} (x\!\cdot\!\nu)\bigl(|\nabla \Phi(u)|^{2}+(\Delta \Phi(u))^{2}\bigr)\,dS .
	$$
Then we deduce
	\begin{equation}\label{eq:Iphi-final}
		\begin{aligned}
			I_{\Phi}
			&= \frac54 \int_{\Omega} |\nabla \Phi(u)|^{2}\,dx
			+ \frac74 \int_{\Omega} (\Delta \Phi(u))^{2}\,dx
			+ \frac12 \int_{\Omega} (x\!\cdot\!\nabla q)\,u^{2} \Phi(u)\,dx \\
			&\quad + \frac14 \int_{\partial\Omega} (x\!\cdot\!\nu)\big(|\nabla \Phi(u)|^{2}+(\Delta \Phi(u))^{2}\big)\,dS .
		\end{aligned}
	\end{equation}
	Finally, inserting \eqref{eq:IDelta}, \eqref{eq:Ip}, \eqref{eq:Ilambda} and \eqref{eq:Iphi-final} into
	\[
	I_{\Delta}+I_{p}+I_{\Phi}+I_{\lambda}=0
	\]
	and moving the boundary terms to the left-hand side yields \eqref{eq:Pohozaev}.
\end{proof}

\begin{lemma}\label{lem:4.4}
	Assume that $\Omega\subset\R^{3}$ is a smooth bounded domain, star-shaped with respect to the origin, fix $\mu>0$ and  $q(x)\in C^{1}(\overline\Omega)$. 
	For each $u\in H_{0}^{1}(\Omega)$, 
	define
	\[
	\begin{aligned}
		\mathcal B(u):={}&-\frac54 \int_{\Omega} |\nabla\Phi(u)|^{2}\,dx
		-\frac74 \int_{\Omega} (\Delta\Phi(u))^{2}\,dx
		-\frac12 \int_{\Omega} (x\!\cdot\!\nabla q)\,u^{2}\Phi(u)\,dx \\[2pt]
		&\quad - \frac14 \int_{\partial\Omega} (x\!\cdot\!\nu)
		\big(|\nabla\Phi(u)|^{2}+(\Delta\Phi(u))^{2}\big)\,dS
		+ 6 \int_{\Omega} q(x)\,u^{2}\Phi(u)\,dx,
	\end{aligned}
	\]
	Then there exists a constant $C_{\mu}>0$ (depending only on $\Omega$, $p$, $|q(x)|_{\infty}$, $\mu$, and $\lambda_{1}$) such that
	\[
	|\mathcal B(u)| \le C_{\mu} \qquad \text{for all } u\in S(\mu).
	\]
\end{lemma}

\begin{proof}
	By Lemma~\ref{lem:Phi-mu}, for every $u\in S(\mu)$ we have
	\[
	\int_{\Omega} \big(|\nabla\Phi(u)|^{2}+(\Delta\Phi(u))^{2}\big)\,dx
	= \|\Phi(u)\|_{\mathbb{H}}^{2}
	\le C_{\Phi}^{2}\,\mu^{2}. 
	\]
	Moreover, since $\Omega$ is smooth and bounded and $x\!\cdot\!\nu$ is bounded on $\partial\Omega$, the trace inequality yields
	\[
	\left|\int_{\partial\Omega} (x\!\cdot\!\nu)\big(|\nabla\Phi(u)|^{2}+(\Delta\Phi(u))^{2}\big)\,dS\right|
	\le C\,\|\Phi(u)\|_{\mathbb{H}}^{2}
	\le C\,C_{\Phi}^{2}\,\mu^{2}, 
	\]
	for a constant $C>0$ depending only on $\Omega$.
	
	For the mixed term we use the embedding $H^{2}(\Omega)\hookrightarrow L^{\infty}(\Omega)$  together with Lemma~\ref{lem:Phi-mu}: there exists $C_{1}>0$ such that
	\[
	|\Phi(u)|_{\infty} \le C_{1}\,\|\Phi(u)\|_{\mathbb{H}} \le C_{1} C_{\Phi}\,\mu,
	\]
	hence
	\[
	\left|\int_{\Omega} (x\!\cdot\!\nabla q(x))\,u^{2}\Phi(u)\,dx\right|
	\le |x\!\cdot\!\nabla q(x)|_{\infty}\,|u|_{2}^{2}\,|\Phi(u)|_{\infty}
	\le C_{2}\,\mu^{3}, 
	\]
	In conclusion, there exists a constant $C_{\mu}>0$ such that
	\[
	|\mathcal{B}(u)| \le C_{\mu} \qquad \text{for all } u \in S(\mu).
	\]
This completes the proof.
\end{proof}

\begin{lemma}\label{lem:4.5}
	Let $\Omega\subset\R^{3}$ be a smooth bounded domain, star-shaped with respect to the origin, and let $q\in\bigl(\tfrac{10}{3},6\bigr)$. 
	If $u\in \mathcal{S}(\mu)$ satisfies $J(u)\le \frac{\mu\,\lambda_{1}}{2}+\frac14\,C_{\Phi}^{2}\mu^{2},$
	then
	\[
	\frac{2\lambda_{1}\,(q-6)}{6\bigl(q-\tfrac{10}{3}\bigr)}
	-\frac{2\,(q-2)\,C_{\mu}}{\mu\,\bigl(q-\tfrac{10}{3}\bigr)}
	\;\le\;
	\lambda_u
	\;\le\;
	\lambda_{1}. 
	\]
\end{lemma}

\begin{proof}
	Since $u\in \mathcal{S}(\mu)$, we have
	\begin{equation}\label{eq:4.7}
		|\nabla u|_{2}^{2}
		= \lambda_u\,\mu - \|\Phi(u)\|_{\mathbb{H}}^{2} + \int_{\Omega} f(u)\,u \,dx. 
	\end{equation}
Using \eqref{eq:4.7} together with assumption~\((f_{4})\), we obtain
	\begin{align}\label{eq:4.8a}
		\frac{\mu\,\lambda_{1}}{2}+\frac14\,C_{\Phi}^{2}\mu^{2}
		&\ge J(u)
		= \frac12 |\nabla u|_{2}^{2}
		+ \frac14 \|\Phi(u)\|_{\mathbb{H}}^{2}
		- \int_{\Omega} F(u)\,dx \notag\\
		&\ge \frac{\mu\,\lambda_u}{2}
		+ \frac12 \int_{\Omega} f(u)\,u\,dx
		- \int_{\Omega} F(u)\,dx
		- \frac14 C_{\Phi}^{2}\mu^{2} \notag\\
		&\ge \frac{\mu\,\lambda_u}{2}
		+ \Bigl(\frac12 - \frac1q\Bigr)\int_{\Omega} f(u)\,u\,dx
		- \frac14 C_{\Phi}^{2}\mu^{2}. 
	\end{align}
	Hence
	\begin{equation}\label{eq:4.8}
		\int_{\Omega} f(u)\,u\,dx
		\le \frac{q\,\mu\,(\lambda_{1}-\lambda_u) + q\,C_{\Phi}^{2}\mu^{2}}{q-2}.
	\end{equation}
For $\mu>0$ sufficiently small, the right-hand side of \eqref{eq:4.8} becomes negative if $\lambda_u>\lambda_{1}$, which is impossible.
	Therefore $\lambda_u \le \lambda_{1}$. 
	
	For the lower bound, by the Poho\v{z}aev-type identity in Proposition~\ref{prop:Pohozaev}.
	Since $\Omega$ is star-shaped, the boundary term in~\eqref{eq:Pohozaev} is nonnegative.
	Rearranging that identity as in the previous step gives
	\[
	\lambda_u
	\;\ge\;
	\frac{(q-6)\,\bigl[(\lambda_{1}-\lambda_u)+C_{\Phi}^{2}\mu\bigr]}{2\,(q-2)}
	\;-\;
	\frac{3}{\mu}\,\mathcal B(u). 
	\]
	By Lemma~\ref{lem:4.4} we have $|\mathcal B(u)|\le C_{\mu}$ for all $u\in S(\mu)$. 
	Since here $p>\tfrac{10}{3}$ and $(f_{4})$ provides $q>\tfrac{10}{3}$, we deduce
	\[
	\lambda_u
	\;\ge\;
	\frac{2\,(\lambda_{1}+C_{\Phi}^{2}\mu)\,(q-6)}{6\,\bigl(q-\tfrac{10}{3}\bigr)}
	\;-\;
	\frac{2\,(q-2)\,C_{\mu}}{\mu\,\bigl(q-\tfrac{10}{3}\bigr)},
	\]
The proof is complete.
\end{proof}

\begin{lemma}\label{cbdd}
	Assume that \emph{($f_1$)}, \emph{($f_2$)} and \emph{($f_4$)} hold, and let 
	$p\in\bigl(\frac{10}{3},6\bigr)$. 
	Then the functional $J(u)$ is bounded from below on the set
	\[
	\Big\{ u\in \mathcal S(\mu):\ 
	J(u)\le \frac{\mu\lambda_1}{2}+\frac14 C_\Phi^2\mu^2\Big\}.
	\]
\end{lemma}
\begin{proof}
	Since $(u)\in\mathcal{S}(\mu)$, it follows that
\begin{equation}\label{eq:omega_identity}
	\int_\Omega |\nabla u|^2\,dx
	-\int_\Omega f(u)u\,dx
	+\frac12\int_\Omega q(x)\Phi(u)\,u^2\,dx
	= \omega \mu^2 .
\end{equation}
	Moreover, by the definition of $J(u)$, we have
\begin{equation}\label{ju}
	J(u)
	= \frac12\int_\Omega |\nabla u|^2\,dx
	- \int_\Omega F(u)\,dx
	+ \frac14\int_\Omega q(x)\Phi(u)\,u^2\,dx .
\end{equation}
	
	Combining \ref{ju} with \ref*{eq:omega_identity}, we obtain
	\[
	J(u)
	= \frac12\int_\Omega f(u)u\,dx
	- \int_\Omega F(u)\,dx
	+ \frac12\,\omega\,\mu^2 .
	\]
	
	Then, by Lemma~\ref{lem:4.5} together with \emph{($f_4$)}, we deduce that
	\[
	J(u)
	\ge \frac12
	\left(
	\frac{2\bigl(1+C_q^2\mu\bigr)(q-6)}{6\left(q-\frac{10}{3}\right)}
	-\frac{2(q-2)C_\mu}{\mu\left(q-\frac{10}{3}\right)}
	\right)\mu^2 .
	\]
	The proof is complete.
\end{proof}

We are now in the position to prove Theorem~\ref{thm:1.4}. 

\begin{proof}
	By Lemma~\ref{cbdd}, let $\{u_n\}\subset \mathcal{S}(\mu)$ be a sequence such that
	\[
J(u_n)\to \inf_{u\in\mathcal{S}(\mu)} J(u),
	\]
	and denote $c:=\inf_{u\in\mathcal{S}(\mu)} J(u).$
	
	For each $n\in\mathbb{N}$, we have
	\[
	J(u_n)
	=\frac12\int_\Omega f(u_n)u_n\,dx
	-\int_\Omega F(u_n)\,dx
	+\frac12\,\lambda_n\,\mu^2 .
	\]
	By \emph{($f_4$)}, it holds that
	\[
	\frac12\int_\Omega f(u_n)u_n\,dx-\int_\Omega F(u_n)\,dx \ge 0.
	\]
Since $J(u_n)\to c$ and
\[
c \le \frac{\mu\lambda_1}{2}+\frac14 C_\Phi^2\mu^2,
\]
it follows from Lemma~\ref{lem:4.5} that the sequence $\{\lambda_n\}$ is bounded from above and below, and hence bounded.
	
	A direct computation shows that
	\[
	c+1+\|u_n\|
	\ge J(u_n)-\frac1q\Big(\langle J'(u_n),u_n\rangle+\lambda_n\mu^2\Big)
	\ge \Big(\frac12-\frac1q\Big)\|u_n\|^2-\frac1q\,\lambda_n\mu^2 .
	\]
Since $\{\lambda_n\}$ is bounded, this implies that $\{u_n\}$ is bounded.
	
	Then, up to a subsequence, there exists $u\in H_0^1(\Omega)$ such that
	\[
	u_n \rightharpoonup u \quad \text{in } H_0^1(\Omega).
	\]	
	Repeating the same argument as in Lemma~\ref{le3.3}, we conclude that, up to a subsequence, $u_n\to u$ strongly in $H_0^1(\Omega)$.
	
	Since $u_n\to u$ strongly in $H_0^1(\Omega)$, for any $v\in H_0^1(\Omega)$ we have
	\[
	\int_\Omega \nabla u_n\cdot \nabla v\,dx \to \int_\Omega \nabla u\cdot \nabla v\,dx \quad \text{as } n\to\infty,
	\]
	\[
	\int_\Omega f(u_n)v\,dx \to \int_\Omega f(u)v\,dx \quad \text{as } n\to\infty,
	\]
	and $\int_\Omega q(x)\Phi(u_n)u_n v\,dx \to \int_\Omega q(x)\Phi(u)u v\,dx \quad \text{as } n\to\infty.$
	Moreover, since $\{\lambda_n\}$ is bounded, up to a subsequence $\lambda_n\to\lambda$ for some $\lambda\in\R$. 
	Consequently,
	\[
	\lambda_n\int_\Omega u_n v\,dx \to \lambda\int_\Omega u v\,dx \quad \text{as } n\to\infty.
	\]
 Then, we obtain that for all $v\in H_0^1(\Omega)$,
\[
\int_\Omega \nabla u\cdot \nabla v\,dx
-\int_\Omega f(u)v\,dx
+\frac12\int_\Omega q(x)\Phi(u)u v\,dx
=\lambda\int_\Omega u v\,dx.
\]
	In conclusion, $u\in\mathcal{S}(\mu)$. Therefore, $u$ is a normalized ground state solution.

\end{proof}

  \end{document}